\documentclass[leqno]{amsart}
\usepackage{fullpage}
\usepackage[english]{babel}
\usepackage{amsmath,amsthm,amssymb,amsfonts,mathtools,mathrsfs}
\usepackage[alphabetic]{amsrefs}
\usepackage[frak=pxtx,cal=euler,scr=boondoxupr]{mathalfa}	
\usepackage{hyperref}
\usepackage{xcolor}	
\usepackage{thmtools} %just for fixing cleveref issues 
\usepackage{cleveref}
\usepackage{enumitem}
\newcommand{\RicE}[1]{{#1}^{-1}\Ric_{#1}}  % (1,1)-Ricci tensor
\DeclareMathOperator\Sym{Sym}
\DeclareMathOperator\Isom{Isom}
\newcommand{\hca}{\mathcal{H}}
\newcommand{\vca}{\mathcal{V}}
\newcommand{\N}{\sans{N}}
\newcommand{\G}{\sans{G}}
\newcommand{\K}{\sans{K}}

\newcommand{\Glp}{{\mathsf{GL}^+}}
\newcommand{\mcv}{\mathrm H}
\newcommand{\pca}{\mathcal{P}}

\newcommand{\ngo}{\mathfrak{n}}
\newcommand{\spacesymbol}[1]{\mathbb{#1}}
\newcommand{\sans}[1]{\mathsf{#1}}
\newcommand{\R}{\spacesymbol{R}}
\newcommand{\C}{\spacesymbol{C}}

\newcommand{\Z}{\spacesymbol{Z}}

\newcommand{\m}[2]{\left\langle #1, #2 \right\rangle}
\DeclareMathOperator\tr{tr}

\DeclareMathOperator\id{Id}
\DeclareMathOperator\dive{div}
\DeclareMathOperator\Ric{Ric}

\DeclareMathOperator\scal{scal}

\DeclareMathOperator\GL{\sans{GL}}
\DeclareMathOperator\SL{\sans{SL}}
\DeclareMathOperator\PSL{\sans{PSL}}
\DeclareMathOperator\Or{\sans{O}}
\DeclareMathOperator\SO{\sans{SO}}
\DeclareMathOperator\Aut{Aut}
\DeclareMathOperator\Aff{Aff}
\DeclareMathOperator\Der{Der}

\DeclareMathOperator\Diff{Diff}
\DeclareMathOperator\Hom{Hom}
\DeclareMathOperator\End{End}

\newcommand{\ad}{{\operatorname{ad}}}
\DeclareMathOperator\Ad{Ad}
\DeclareMathOperator\sym{Sym}

\AtBeginDocument{%
   \def\MR#1{}
}

\newtheorem{maintheorem}{Theorem}

\newtheorem{maincorollary}[maintheorem]{Corollary}

\newtheorem{theorem}{Theorem}[section]

\newtheorem{definition}[theorem]{Definition}
\newtheorem{prop}[theorem]{Proposition}
\newtheorem{corollary}[theorem]{Corollary}
\newtheorem{lemma}[theorem]{Lemma}

\theoremstyle{definition}
\newtheorem{remark}[theorem]{Remark}

\newtheorem{ansatz}[theorem]{Ansatz}

\theoremstyle{remark}
\newtheorem{example}[theorem]{Example}

\title{Expanding Ricci solitons and Higgs bundles}

\author[R.~A.~Lafuente]{Ramiro A.~Lafuente}
\address{School of Mathematics and Physics, The University of Queensland, St Lucia QLD 4072, Australia}
\email{r.lafuente@uq.edu.au}

\author[A.~Thompson]{Adam Thompson}
\address{University of M\"unster, Einsteinstrasse 62, 48149 M\"unster, Germany}
\email{a.thompson@uni-muenster.de}

\begin{document}

\begin{abstract}
Motivated by the long-time behavior of Ricci flows that collapse with bounded curvature, we study expanding Ricci solitons with nilpotent symmetry on vector bundles over a closed manifold. We prove that, under mild assumptions that are satisfied by Ricci flow limits, the equations dimension-reduce to the so-called twisted harmonic-Einstein equations. When the base is a surface, we establish a correspondence between solutions of the latter and a class of $\G$-Higgs bundles. This allows us to produce infinite families of new examples that are not locally homogeneous, and in particular to obtain a complete description in dimension $4$. We also show that all our examples admit Einstein one-dimensional extensions. 
\end{abstract}

\maketitle

\setcounter{tocdepth}{1}
\tableofcontents

\section{Introduction}\label{Sec:Intro}

An expanding Ricci soliton is a triple $(M,g,X)$ consisting of a complete Riemannian manifold  and a vector field that satisfy 
 \begin{equation}\label{eqn_eRS}
 	\Ric(g)  + \tfrac{1}{2} \mathcal{L}_X g = - \tfrac12 \, g , 
 \end{equation}
 where $\mathcal{L}$ denotes the Lie derivative. Solutions to \eqref{eqn_eRS} give rise to self-similar solutions of the Ricci flow: if $(\eta_t)_{t\in\R}$ denotes the flow of $X$, then $g(t) := \eta_t^* g$ solves the normalised Ricci flow equation
\begin{equation}\label{eqn:NRF}
		\partial_t \, g = - 2  \Ric(g)   - \, g.
\end{equation}
(Notice that $(g(t))_{t\in \R}$ solves \eqref{eqn:NRF} if and only if $(\bar g(s) := s \cdot g (\log(s)))_{s\in (0,\infty)}$ solves the unnormalised  flow.) Besides generalising Einstein manifolds, and modelling desingularisation of cone-like singularities of the flow \cites{Der16,GS18}, these solitons also play a central role in understanding the long-time behavior of  Ricci flow solutions that exist for all $t\in (0,\infty)$ (\emph{immortal} solutions). As shown by the 3D case \cite{Lott10}, solitons with constant scalar curvature are particularly relevant. However, as far as we know, all known examples of expanding Ricci solitons with constant scalar curvature are locally homogeneous.

Our aim in this article is to construct, in each dimension $n\geq 4$, infinite families depending on an arbitrary large number of parameters, of expanding Ricci solitons with constant negative scalar curvature that are \emph{not} locally homogeneous.  
% Since they have constant negative scalar curvature, these solitons are \emph{not} asymptotically conical. 
% \begin{theorem}\cite{LP25}\label{thm_LP}
% Let $(M^n, g(t))_{t\in (0,\infty)}$, $n\geq 4$, be an immortal solution to  \eqref{eqn:NRF} with uniformly bounded curvature and diameter. If $n\geq 5$, assume in addition that $M$ is essential. Then, as $t\to \infty$, $g(t)$ subconverges to an expanding Ricci soliton with local nilpotent symmetry that fibres over a closed orbifold $B$. 
% \end{theorem}  
% This generalises to higher dimensions a  result of Lott in dimension $3$ \cite{Lott10}. 
% In general, the Gromov-Hausdorff limit $B$ will have dimension strictly less than $n$, in which case the sequence collapses with bounded curvature. This is the reason why the limits have nilpotent symmetry (cf.~\cite{CFG}). The convergence to an $n$-dimensional smooth limit soliton has to be understood either in the sense of \'etale Riemannian groupoids  \cite{Lot07}, or in the sense of  unwrapped local neighborhoods and $N^*$-structures \cites{NT11,NT18}. 
The underlying manifolds  are  diffeomorphic to a so-called \emph{twisted principal bundle}: a Lie group bundle
\begin{equation}\label{eqn:tpb_intro}
	M^n = \tilde B \times_\rho \N \longrightarrow B^2,
\end{equation}
over a closed surface, with fibre a simply-connected nilpotent Lie group $\N$. More precisely, $M$ is the total space of the associated bundle to the universal cover $\tilde B \to B$  corresponding to a representation $\rho\in \Hom(\pi_1(B), \Aut(\N))$; see \Cref{Sec:Twisted Bundles} for  details. 
% This is a fibre bundle, with fibre  a simply-connected nilpotent Lie group $\N$, and structure group contained in $\Aut(\N)\ltimes \N$, together with  a flatness condition. Here $\tilde B$ is the universal cover and $\rho\in \Hom(\pi_1(B), \Aut(\N))$ is a representation. Flatness means that \eqref{eqn:tpb_intro} admits a bundle atlas of local trivialisations with locally constant transition functions. In each such trivialisation, there is a free $\N$-action which is simply-transitive on the fibres and preserves the Ricci soliton metric. 
Locally, there is a free $\N$-action that is transitive on fibres, which in general cannot be extended to a well-defined action on the total space due to the global twist by $\rho$. 

It turns out that the solitons in this article have additional structure: they are all \emph{affine expanding solitons}. In essence, this means that they are complete, the local $\N$-actions are isometric, and the diffeomorphisms $\eta_t : M \to M$ driving the Ricci flow evolution are bundle maps that act as automorphisms on the fibres: see Definition \ref{def_affinesoliton} for a precise definition. The name is motivated by a canonical affine connection on the fibres of $M$ which encodes the local $\N$-actions (see $\S$\ref{sec:can_vert}), and is key in the collapsing theory \cite{CFG}.  These solitons  generalise in a natural way the well-known \emph{algebraic solitons} from the homogeneous setting \cites{homRS, Jbl2015,Jbl13b}, and thanks to  a forthcoming result \cite{LP25},  they indeed describe the general asymptotic behavior of immortal Ricci flows on \emph{closed} manifolds that collapse with bounded cuvature and diameter.

% and the geometry is that of a principal bundle with invariant metric. However, the global twist by $\rho$ generally prevents said action to glue up to a well-defined action on the total space.
% % Remarkably, apart from the locally homogeneous examples and the work of Lott \cite{Lott10}, to the best of our knowledge no general construction of expanding Ricci solitons of this sort exists to this date. 

\smallskip

A second motivation
comes from the recent \emph{non-existence} results for invariant Einstein metrics of negative scalar curvature on (untwisted) principal bundles, see \cites{alek_sol,BL23b}. As an upshot of  the present paper, a plethora of examples can be constructed if one allows for the isometric actions to be defined only locally.

% {\alert Where to say this?} These solitons are typically of non-gradient type ($X \neq \nabla f$).

\bigskip

Our first main result is the following dimension-reduction theorem, which makes no assumption on  $\dim B$:

\begin{maintheorem}\label{main:eRS-tHE}
Let $B$ be a closed manifold and $\N$ a simply-connected nilpotent Lie group. Any affine expanding Ricci soliton on a twisted principal bundle $\tilde B \times_\rho \N$  gives rise to a solution to the twisted harmonic-Einstein equations \eqref{eqn_twistharmEin} for the pair $(B, \G/\K)$, where $\G/\K$ is a symmetric space of non-compact type naturally associated to $\N$. Conversely, to each solution to \eqref{eqn_twistharmEin} there corresponds an affine expanding Ricci soliton on $\tilde B \times_\rho \N$.
\end{maintheorem}

% \begin{maintheorem}\label{main:eRS-tHE}
% Let $B$ be a closed manifold and $\N$ a simply-connected nilpotent Lie group. Any invariant expanding Ricci soliton on a twisted principal bundle $\tilde B \times_\rho \N$ that arises as a Ricci flow limit as per \Cref{thm_LP} gives rise to a solution to the twisted harmonic-Einstein equations \eqref{eqn_twistharmEin} for the pair $(B, \G/\K)$, where $\G/\K$ is a symmetric space of non-compact type naturally associated to $\N$. Conversely, to each solution to \eqref{eqn_twistharmEin} there corresponds an invariant expanding Ricci soliton on $\tilde B \times_\rho \N$.
% \end{maintheorem}

The \emph{twisted harmonic-Einstein equations} for a pair $(B, \G/\K)$ consisting of a smooth manifold $B$ and a symmetric space of non-compact type $\G/\K$  is the following coupled system of equations for a triple $(g_B, \rho, h)$, where $g_B$ is a Riemannian metric on $B$, $\rho \in \Hom(\pi_1(B),\G)$, and $h : \tilde B \to \G/\K$ is smooth and $\rho$-equivariant: 
\begin{equation} \label{eqn_twistharmEin}
\begin{cases}
 \tr_{\tilde g_B} (\nabla d h) &= 0 \, ,  \\
 % \label{Eq:twistharmEinV} \\
 \Ric(\tilde g_B) &= - \tfrac12 \tilde g_B+h^*g_{\mathsf{sym}}. 
 % \label{Eq:twistharmEinH}
\end{cases}
\end{equation}
Here $\tilde g_B$ is the lift of $g_B$ to $\tilde B$. The first equation
is simply the harmonic map equation for $h : (\tilde B, g_B) \to (\G/\K, g_{\mathsf{sym}})$ \cite{ES64}, whereas the second one is a variation of the Einstein equation. These equations first appeared in \cite{Lot07}, see also \cites{List08,Mul12}. A similar dimensional reduction of the Ricci-flat equation to a coupled equation for a harmonic map and metric on a lower dimensional manifold is exploited in the study of toric gravitational instantons \cites{KL21,LS25}. The way an invariant expanding soliton gives rise to a solution of \eqref{eqn_twistharmEin} is as follows: $g_B$ is the unique metric on $B$ such that \eqref{eqn:tpb_intro} is a Riemannian submersion; the connection defined by the horizontal distribution is flat (i.e.~said distribution is integrable), and its monodromy is $\rho$;  and $h$ encodes the fibre metrics, which may be viewed as left-invariant metrics on $\N$.

We notice that, in the particular case of invariant Einstein metrics on principal bundles, \Cref{main:eRS-tHE} essentially reduces to  \cite{alek_sol}*{Thm.~G}, which is the key step towards the proof of the Alekseevskii conjecture.
The converse assertion in \Cref{main:eRS-tHE} uses an Ansatz (\ref{Ansatz}) inspired by the work of Lott, who proved \Cref{main:eRS-tHE} for abelian $\N$ \cites{Lot07,Lott10}. When $\N$ is not abelian, the expanding solitons constructed in \Cref{main:eRS-tHE} are \emph{non-gradient}, i.e.~$X$ cannot be taken as the gradient of a smooth function on $M$ (\Cref{prop:non-gradient}).

In order to describe the symmetric space $\G/\K$ associated to $\N$, we  first mention that a necessary condition for solving \eqref{eqn_eRS} on $\tilde B \times_\rho \N$ is that $\N$ admits a \emph{nilsoliton} metric $\bar h$, i.e.~a left-invariant expanding Ricci soliton \cite{soliton}, whose Ricci curvature (in a left-invariant frame) satisfies $\Ric_{\bar h} = -\tfrac12 \bar h + \bar h(D \cdot, \cdot)$, for some derivation $D\in \Der(\ngo)$. We then set
\begin{equation}\label{eqn_G=Autbeta}
	\G := 
	% \big({\Aut(\sans{N})^\beta}\big)^0 = 
	\big\{\varphi\in\Aut(\sans{N}) :  \varphi_* D \varphi_*^{-1} = D,
    % \varphi^* \big(\RicE{\bar h}\big) = \RicE{\bar h} , 
    % \,\, \varphi^* |{\rm d} \mu_{\bar h}| = |{\rm d} \mu_{\bar h}| \big\}, \qquad \K := \G \cap \Isom(\N, \bar h).
    \,\, \det \varphi_* = \pm 1 \big\}, \qquad \K := \G \cap \Isom(\N, \bar h),
\end{equation}
where 
% $|d\mu_{\bar h}|$ is the Riemannian volume density associated to $\bar h$, and 
$\varphi_* = d\varphi|_e \in \Aut(\ngo)$.
Results from real GIT imply that $\K$ is a maximal compact subgroup of the real reductive Lie group $\G$ \cites{RS90,Nik11}, see \Cref{app_nil}. Examples of nilsolitons include flat Euclidean space $\R^n$, and (the universal cover of) the ${\rm Nil^3}$ geometry in dimension three. If $\N\simeq \R^n$ is abelian, $D = \tfrac12 {\rm Id}$ and we have $\G = \det^{-1}(\pm 1) \subset \GL_n(\R)$, $\K = \mathsf{O}(n)$, thus $\G/\K$ is the space of inner products on $\R^n$ with a fixed density. For non-abelian $\N$ we may always view $\G/\K$ as a totally geodesic submanifold of $\det^{-1}(\pm 1) / \mathsf{O}(n)$.

\bigskip

In our second main result, we apply Theorem~\ref{main:eRS-tHE} to produce large families of Ricci solitons by solving the twisted harmonic-Einstein equations \eqref{eqn_twistharmEin}. It turns out that all expanding solitons corresponding to solutions with $\dim B \leq 1$ are locally homogeneous (see \Cref{sec:dimB=1}), so the first interesting case for us is when $B$ is a closed surface. In this case, a result of Labourie \cite{Labourie08} allows us to solve the twisted harmonic-Einstein equations for a certain class of representations.

%use a result of Labourie \cite{Labourie08} to solve the twisted harmonic-Einstein equations \eqref{eqn_twistharmEin} for a certain . This immediately yields large families of cohomogeneity two Rocco solitons. 
%We now apply Theorem~\ref{main:eRS-tHE}, together with , to solve . It turns out that all expanding solitons corresponding to solutions with $\dim B \leq 1$ are locally homogeneous (see \Cref{sec:dimB=1}), so the first interesting case is when $B$ is a closed surface.  In this direction, o uses Theorem~\ref{main:eRS-tHE}  to obtain the following positive answer to the existence question :
%If instead of fixing the conformal structure, we let it vary and fix instead the representation $\rho$, a result of Labourie \cite{Labourie08} implies that, for a certain class of representations, we obtain the following positive answer to the existence question:

\begin{maintheorem}\label{maincor_Hit_comp}
Let $B^2$ be a closed surface, and suppose that $\rho \in \Hom(\pi_1(B), \SL_n(\R))$ lies in a Hitchin component. 
% Let $\N$ be a nilpotent Lie group admitting a nilsoliton metric. 
Then, on the $\rho$-twisted principal bundle $M = \tilde B \times_\rho \R^n$ there exists an affine expanding Ricci soliton metric.
\end{maintheorem}

%is a lift of a representation $\check \rho \in \Hom(\pi_1(B),\PSL_n(\R))$ whose class $[\check \rho] \in \Hom(\pi_1(B), \mathsf{PSL}_n(\R))/\mathsf{PSL}_n(\R)$

The Hitchin components, discovered in \cite{Hit92}, are certain distinguished connected components of the space of representations $\Hom(\pi_1(B), \G) / \G$ for an adjoint split semisimple Lie group $\G$, which in some sense generalise the Teichm\"uller space component in the case $\G = \mathsf{PGL}_2(\R)$. If \(\G\) is a split semisimple Lie group that is not necessarily adjoint type, we abuse terminology by saying \(\rho \in \Hom(\pi_1(B), \G)\) lies in a Hitchin component if the conjugacy class of \(\Ad\circ \rho\) does. Hitchin components are homeomorphic to  open balls of dimension $(2g-2) \dim \G$, where $g$ is the genus of $B$. Since the underlying manifold  $M^{n+2} = \tilde B \times_\rho \N$ remains constant when $\rho$ varies in a fixed component, \Cref{maincor_Hit_comp} yields the existence of a  $(2g-2)(n^2-1)$-parameter family of complete expanding Ricci solitons on $M^{n+2}$ admitting a free, local, isometric $\R^n$-action.

%We have slightly abused terminology in Theorem~\ref{maincor_Hit_comp} and called components of $\Hom(\pi_1(B), \SL_n(\R))$ Hitchin components if they map to Hitchin components under the map $\Hom(\pi_1(B), \SL_n(\R))\to \Hom(\pi_1(B), \mathsf{PSL}_n(\R))/\mathsf{PSL}_n(\R) $ obtained by composing the homomorphism \(\SL_n(\R)\to \PSL_n(\R)\) with the quotient map \(\Hom(\pi_1(B), \mathsf{PSL}_n(\R))\to \Hom(\pi_1(B), \mathsf{PSL}_n(\R))/\mathsf{PSL}_n(\R)\).

\smallskip

If we fix a conformal structure \(c\) on the closed surface \(B^2\), we are able to exhibit more solutions to the twisted harmonic-Einstein equations by exploiting the non-abelian Hodge correspondence: 

\begin{maintheorem}\label{main:tHE-Higgs}
Let $c$ be a conformal class on a closed  surface $B^2$ of genus at least $2$. Then, there is a one-to-one correspondence between equivalence classes of solutions $(g_B, \rho, h)$ to the twisted harmonic-Einstein equations \eqref{eqn_twistharmEin} with $g_B \in c$, and equivalence classes of polystable $\G$-Higgs bundles $(P,\Phi)$ over $(B,c)$ with $\tr \Phi^2 = 0$ up to pullback by automorphisms of \((B^2,c)\).
\end{maintheorem}

Higgs bundles were introduced by Hitchin \cite{Hitchin87} and Simpson \cite{Simp88}, and have been extensively studied since then. We refer the reader to \Cref{sec_riemsurf} below for their definition and related notions such as stability and equivalence. The equivalence notion for solutions of \eqref{eqn_twistharmEin} is motivated by the link with Riemannian metrics on twisted principal bundles, and is made precise in \Cref{def:tHE_equivalence} (see also \Cref{prop:param_wilson}).   We emphasise that \Cref{main:tHE-Higgs} holds for any symmetric space $\G/\K$ of non-positive curvature, and is not restricted to those arising from the soliton construction as in \eqref{eqn_G=Autbeta}.

Theorems \ref{main:eRS-tHE} and \ref{main:tHE-Higgs} yield the following consequence, which is our main tool for producing examples of expanding Ricci solitons by using Higgs bundles:

\begin{maincorollary}\label{maincor:isom_of_sols}
Given a closed Riemann surface $\Sigma = (B^2, c)$ of genus at least $2$, there is a finite-to-one correspondence between polystable  $\G^0$-Higgs bundles $(P,\Phi)$ over $\Sigma$ with $\tr \Phi^2= 0$, up to equivalence and pullback by $\Aut(\Sigma)$, and affine expanding Ricci solitons on twisted principal bundles $\tilde B \times_\rho \N \to B$ with \(\rho(\pi_1(B))\le \G^0\) and inducing the conformal structure $c$ on $B$, up to isometric bundle maps. Here, \(\G^0\) is the identity component of the group \(\G\) defined in \eqref{eqn_G=Autbeta}. 
\end{maincorollary}

The restriction to solitons with monodromy into the identity component $\G^0$ in \Cref{maincor:isom_of_sols} is due to a technicality in the literature on the non-abelian Hodge correspondence, and we do not expect it to be  essential. We note that, up to passing to a finite cover $\hat B \to B$ and pulling back the bundle $M\to B$ to $\hat B$, the assumption \(\rho(\pi_1(B))\le \G^0\) is always satisfied; see the discussion in \Cref{sec:solitons_higgs}. 

The Higgs field $\Phi$ is intimately related to $d h$, and the condition $\tr\Phi^2 = 0$ is well-known to be equivalent to $h$ being weakly conformal (also called a branched minimal immersion). In the special case  where $dh$ has full rank, $h$ is indeed a minimal immersion of $\tilde B$ into the symmetric space $\G/\K$. This condition plays a key role in Labourie's conjecture \cite{Labourie08}, known to hold in rank 2 \cites{Lab17,CTT19}, but recently disproved for higher rank groups \cite{SS22}. See also \cites{AL18,LM19}.

\smallskip

We now give a more detailed description of our results in the particularly relevant case of dimension $4$. If $\dim \N =0$, then the expanding soliton is a closed negative Einstein manifold $(B^4,g_B)$. For $\dim \N = 1$, the soliton is locally isometric to the product of $\R$ and a closed hyperbolic manifold $(B^3, g_B)$ (cf.~\cite{Lott10}*{Prop.~4.80}). As mentioned above, $\dim B \leq 1$ yields locally homogeneous expanding solitons. We are left with  the case $\dim B = \dim \N = 2$. Here, \Cref{maincor:isom_of_sols} and Hitchin's explicit description of the moduli space of $\SL_2(\R)$-Higgs bundles \cite{Hitchin87}*{$\S$10} yield a complete description of 4D affine expanding solitons, which we detail in \Cref{sec:4D}.
% (recall that all Ricci flow limits as in \Cref{thm_LP} do satisfy said Ansatz thanks to \Cref{thm_afine_then_ansatz}). 
In particular, we obtain the following:

\begin{maincorollary}\label{maincor:dim4}
Let $B^2$ be a closed surface of genus $g\geq 3$. Then, there exists a family, depending on $10 g - 14 = 4 (g-2) + 6(g-1)$ real parameters, of affine expanding Ricci solitons with constant scalar curvature on twisted principal bundles $M^4 = \tilde B \times_\rho \R^2$ over $B$, none of which is gradient or locally homogeneous.
% $\G = \sans{Sl}_2(\R)$ ($g\geq 3$).
\end{maincorollary}

For genus $g=2$ we still get a $6 = 6(g-1)$-dimensional family of examples, reflecting the different choices of conformal structures on $B$. However, they are all locally homogeneous (see \Cref{ex:uniformisingrep}).

Regarding non-abelian nilpotent groups, the simplest example is the 3-dimensional Heisenberg group $\N^3$.  Due to the structure of automorphisms preserving the nilsoliton derivation, it turns out that all expanding Ricci solitons on twisted principal bundles $M^5 = \tilde B^2 \times_\rho \N^3$ arise as central extensions of the four-dimensional examples with $\R^2$-symmetry that correspond to $\SL_2(\R)$-Higgs bundles: 

\begin{maincorollary}\label{maincor:dim5}
Let $\N^3$ be the Heisenberg group. Up to isometric bundle maps, the affine expanding Ricci solitons on twisted principal $\N^3$-bundles $M^5 \to B^2$ and having orientation-preserving monodromy, are in  $1$-to-$1$ correspondence with those constructed on twisted principal $\R^2$-bundles $M_0^4 \to B$. 
\end{maincorollary}

The monodromy assumption is always satisfied, after possibly pulling back the bundle to a double-cover $\hat B \to B$.  We have not  specifically pursued the question of which symmetric spaces $\G/\K$ may arise from a nilsoliton Lie group $\N$ via \eqref{eqn_G=Autbeta}. But let us  mention that, by simply restricting to the class of H-type nilpotent Lie groups  \cites{Kap80,Kap81} -- which are known to admit nilsoliton metrics \cite{soliton}--  one already gets a very rich class of examples. Indeed, from the explicit description of $\G$ that may be found in \cite{Saal96} (albeit without the unimodularity constraint), it follows that the symmetric spaces arising from H-type nilpotent Lie groups  contain irreducible factors from 8 out of the 10 infinite families in the classification (the missing ones are $\SL_n(\C)/\mathsf{SU(n)}$ and $\mathsf{SU}(p,q)/{\sans S}(\mathsf{U}(p) \times \mathsf{U}(q))$).

\smallskip

Regarding other non-gradient Ricci solitons, up to very recently, the homogeneous ones were the only source of examples. First constructed by Lauret in \cite{soliton} (see also \cites{BD07,Lot07}), they have been quite extensively studied in the last 20 years, see \cites{cruzchica,Jbl2015,alek_sol} and the references therein. Beyond the homogeneous case, some non-gradient expanders asymptotic to Euclidean cones were obtained in \cite{KL12}. More recently, examples on surfaces were constructed in \cite{TY24}. In dimension 4, asymptotically conical non-gradient expanders with non-negative scalar curvature were shown to exist in \cite{BC23}. The literature on asymptotically conical expanding Ricci solitons is enormous, and we refer the reader to \cite{BC23} and the references therein for an appropriate account of recent developments.

\smallskip

Our last main result is inspired by Lauret's work in the homogeneous case \cite{soliton} relating nilsolitons to Einstein solvmanifolds in one dimension higher, cf.~also \cites{Heber1998,alek,HePtrWyl}.
% It asserts the existence of Einstein manifolds that are naturally associated to the expanding Ricci solitons constructed in this paper. 

\begin{maintheorem}\label{main_Einstein}
If $(M^n, g, X)$ is an affine expanding Ricci soliton,
% constructed. %from \Cref{Ansatz} 
there exist diffeomorphisms $(\phi_u)_{u\in \R}$ of $M$ such that $(\R \times M, g_E :=  du^2 + \phi_u^* g)$ is  Einstein with $\Ric_{g_E} = -\tfrac12 g_E$.
\end{maintheorem}

\medskip

The article is organised as follows. In \Cref{Sec:Twisted Bundles} we introduce twisted principal bundles, focusing on their global properties and the geometry of invariant metrics. We also characterise bundle isometries between them through a parametrised version of Wilson's theorem on the isometries of nilmanifolds (\Cref{prop:param_wilson}). 

In \Cref{Sec:Ansatz} we introduce \Cref{Ansatz} for constructing invariant expanding Ricci solitons on twisted principal bundles, and prove that, in this context, the soliton equation reduces to the twisted harmonic-Einstein equations (\Cref{prop_AnsatthenHE}).  The need for considering twisted principal bundles (as opposed to a globally-defined free isometric Lie group action) is justified by \Cref{Prop:UntwistedSolitons}. 

In \Cref{sec_riemsurf} we restrict ourselves to the case $\dim B = 2$, and prove \Cref{maincor_Hit_comp}. We then 
% \Cref{main:tHE-Higgs}. To that end, after proving \Cref{maincor_Hit_comp}, we 
introduce $\G$-Higgs bundles,  discuss the non-abelian Hodge correspondence, and prove \Cref{main:tHE-Higgs}. The proof of \Cref{maincor:isom_of_sols} is given in \Cref{sec:solitons_higgs}. The cases of dimensions $4$ and $5$, including the proofs of Corollaries \ref{maincor:dim4} and \ref{maincor:dim5}, are discussed in \Cref{sec:4D}.

In \Cref{Sec:ExpandingSolitons} we introduce affine expanding solitons, and prove that they
% show that any expanding soliton arising as a limit of collapsing immortal Ricci flows with bounded curvature and diameter as per \Cref{thm_LP} 
must indeed satisfy \Cref{Ansatz} (\Cref{thm_afine_then_ansatz}). This, together with the  results in \Cref{Sec:Ansatz},  finishes the proof of \Cref{main:eRS-tHE}. We also characterise the gradient condition (\Cref{prop:non-gradient}), showing in particular that a large class of expanders constructed in this paper are not gradient. 

The proof of the Einstein extension theorem \Cref{main_Einstein} is given in \Cref{sec:Einstein}. Finally, an appendix includes basic facts on nilsolitons that we use throughout the article.

\medskip

\noindent {\itshape Acknowledgements.\ }  The first-named author thanks John Lott for helpful comments about twisted principal bundles, and Peter Gothen and Nicholas Tholozan for their help with the theory of Higgs bundles. The second-named author thanks the organisers of the Institut Henri Poincar{\'e} program `Higher rank geometric structures' for inviting him to attend the program, and Oscar Garc{\'i}a-Prada for explaining to him the subtleties of Higgs bundles with disconnected structure groups. 

RL was supported by The University of Queensland through a Foundation Research Excellence Award, and by the Australian Research Council projects DP240101772 and FT240100361. AT was funded by the Deutsche Forschungsgemeinschaft (DFG, German Research Foundation) under Germany's Excellence Strategy EXC 2044 –390685587, Mathematics Münster: Dynamics–Geometry–Structure and by an Australian Government Research Training Program (RTP) Scholarship. Part of this research was carried out during a meeting at the  MATRIX institute in Creswick, Australia, and we are grateful for the hospitality and optimal working conditions.

\medskip

\section{Twisted Principal Bundles}\label{Sec:Twisted Bundles}

In this section we define a class of manifolds with local symmetries that will serve as the underlying spaces for the Ricci solitons we construct. These spaces generalise the notion of a principal bundle, allowing for more phenomena to occur. Our treatment is based on \cite{Lott20} (see also \cite{Lott10}*{$\S$4.1}).

\begin{definition}\label{def_tpb}
Let $\sans{N}$ be a  connected Lie group and $B$ a smooth manifold. 
A \emph{twisted principal $\sans{N}$-bundle} is a fibre bundle \(\pi:M\to B\) with fibre \(\sans{N}\) and  structure group contained in \(\Aut(\N)_\delta\ltimes \sans{N}_L\). (Here, the subscript $\delta$ means that the group $\Aut(\N)$ is endowed with the discrete topology, and $\sans{N}_L \leq \Diff(\N)$ denotes the action of $\N$ on itself  by left multiplication.)
\end{definition}

Notice that, when the structure group is contained in $\N_L$,  one recovers the definition of a principal $\N$-bundle. In that case, there is a natural free right $\N$-action which is transitive on fibres. A similar thing happens for a twisted principal bundle $\pi: M \to B$. Consider a covering $\{U_i\}$ of $B$ and local trivialisations 
\[
    \psi_i : M_{U_i} := \pi^{-1}(U_i) \to U_i \times \N.
\]
For each $i$ there is an obvious right free \(\sans{N}\)--action  coming from right-multiplication on the second factor:
\begin{equation}\label{eqn:cdot_i}
\cdot_i : M_{U_i} \times \N \to M_{U_i}, \qquad 
    \psi_i^{-1}(b,n) \cdot_i n' \mapsto \psi_i^{-1}(b, n \cdot n'). 
\end{equation}
We denote by $\tau_{ij} : U_{ij} \to \Aut(\N)_\delta\ltimes \sans{N}_L$ the transition function on a connected component $U_{ij}$ of $U_i\cap U_j \neq \emptyset$:
\[
     \psi_j \circ \psi_i^{-1} (b,n)  = (b,(\tau_{ij}(b)) (n) ), \qquad b\in U_{ij}, \quad n\in \N.
\]
Since  $\N_L$ commutes with right-multiplication, there exists $\varphi_{ij} \in \Aut(\N)$ such that
\begin{equation}\label{eqn_varphiij}
	p\cdot_i n = p\cdot_j \varphi_{ij}(n),\qquad n\in \sans{N}, \qquad p \in M_{U_{ij}}.
\end{equation}
The fact that $\varphi_{ij}$ is independent of $b$ is  due to  the choice of discrete topology  on $\Aut(\N)$. The upshot is that, even though the action of $\N$ is not globally well-defined, there are well-defined transformation groups $\N_b \subset \Diff(M_b)$ of each fibre $M_b = \pi^{-1}(b)$ that act simply-transitively. These $\N_b$ vary smoothly with $b$, in the sense that they come from transformation groups $\N_U \subset \Diff(M_U)$. An isomorphism $\N_b \simeq \N$ depends of course on the choice of local trivialisation, and for different choices, two such isomorphisms differ by  $\varphi_{ij} \in \Aut(\N)$. 

\begin{definition}\label{def_invariant}
A tensor field on $M$ is called \emph{invariant}, if in local trivialisations $U_i \times \N$ the corresponding tensor is invariant under \eqref{eqn:cdot_i}.
\end{definition} 
Since the local actions yield a well-defined transformation group $N_{U_i}$ of $M_{U_i} = \pi^{-1}(U_i)$ independent of the trivialisation,  invariance is indeed well defined.

\smallskip

From Definition \ref{def_tpb} and the general theory of fibre bundles, any twisted principal bundle is  associated  to some $(\Aut(\sans{N})\ltimes \sans{N}_L)$-principal bundle $\mathcal{P} \to B$, 
whose transition functions have locally constant $\Aut(\sans{N})$-component. The homomorphism $\Aut(\sans{N})\ltimes \sans{N} \to \Aut(\sans{N})$ associates to $\mathcal{P}$ a principal $\Aut(\sans{N})$-bundle  $P \to B$ with
locally constant transition functions, and thus a canonical flat connection with monodromy 
\[
    \rho : \pi_1(B) \to \Aut(\N)
\]
(which is only well-defined up to conjugation).
Notice that the structure group of $M$ is thus contained in $\rho(\pi_1(M))_\delta \ltimes \N_L$. This \emph{twisting homomorphism} $\rho$ measures the failure of the above local $\sans{N}$-actions \eqref{eqn:cdot_i} to glue up to a global $\sans{N}$-action on $M$. Indeed, if $\rho$ is trivial, then $P$ is isomorphic to the trivial bundle $B \times \Aut(\sans{N})$ with the flat connection corresponding to the trivial connection (cf.~\cite{KobayashiNomizu63}*{Ch.~II, Coro.~9.2}, noticing that trivial holonomy is enough). Using this isomorphism one can easily construct local trivialisations for $P$, and hence also for $\mathcal{P}$ and $M$, with transition functions having no $\Aut(\sans{N})$-component. That is, $\pi: M \to B$ is a principal $\sans{N}$-bundle.

If $f: B' \to B$ is a smooth map, the pullback bundle $f^* M \to B'$ is again a twisted principal $\N$-bundle, as it shares the same structure group. The twisting homomorphism $\rho'$ may differ though: say $B' = \tilde B$ is the universal cover of $B$, and let $\tilde \pi : \tilde M \to \tilde B$ denote the pullback bundle. Then $\rho'$ is trivial, hence the bundle is principal. In this way, we may view any twisted principal bundle as a quotient $M = \tilde M / \pi_1(B)$ of a principal $\N$-bundle, where $\pi_1(B)$ acts on $\tilde M$ by \((b,x)\cdot \gamma = (b \cdot \gamma,\rho(\gamma^{-1})x)\).

\begin{example}[The M\"obius band]
The simplest example of a (non-principal) twisted principal bundle is the M\"obius band, i.e.~the non-trivial line bundle $M^2 \to S^1$. The twisting homomorphism maps the generator of $\Z \simeq \pi_1(S^1)$ to the map $x\mapsto -x$. 
\end{example}

\begin{example}[Lie group quotients]\label{ex:RxN}
Generalising the previous example, let $\G = \R \ltimes_\theta \N$ be a semidirect product of Lie groups defined by a homomorphism $\theta : \R \to \Aut(\N)$. There is a properly discontinuous action of $\Z$ on $\G$ given by
\[
    k \cdot (t, n) = (t + 2 \pi k, \theta(2\pi k) n),
\]
thus $M := \G/\Z $ is a smooth manifold. The projection $\G \to \G/\N \simeq \R$ induces a well-defined submersion 
\[
\pi : M \to S^1, \qquad [(t,n)]_M \mapsto [t]_{\R/\Z}.
\]
Then $M$ is a twisted principal $\N$-bundle, with twisting homomorphism $\rho : \pi_1(S^1)\simeq \Z \to \Aut(\N)$, $\rho(k) = \theta(2 \pi k)$.
\end{example}

\subsection{The canonical vertical connections}\label{sec:can_vert}

A key property of twisted principal bundles is the existence of canonically-defined affine connections $\nabla^{\vca,b}$ on the tangent bundle of each fibre $M_b$, that depend smoothly on $b\in B$, and are flat and with parallel torsion.  These connections play a key role in the theory of collapsing Riemannian manifolds with bounded curvature \cites{Fu89,CFG}. 

The construction is as follows (see \cite{CFG}{$\S\S$3--4} for further details). Recall that on the tangent bundle of a simply-connected Lie group $\N$ there is a canonical affine connection $\nabla^{\rm R}$  defined by declaring  the right-invariant vector fields to be parallel. (In \cite{CFG} they use $\nabla^{\rm L}$; our treatment follows the conventions in \cite{NT18}, with the local isometric action being by $\N_R$.) The curvature of $\nabla^{\rm R}$ vanishes thanks to the Jacobi identity, and its torsion is clearly parallel. Conversely, if a simply-connected smooth manifold admits a flat affine connection with parallel torsion, then it is diffeomorphic to a Lie group \cite{KT68}.  The group of diffeomorphisms of $\N$ preserving $\nabla^{\rm R}$ is $\Aff(\N,\nabla^{\rm R}) = \Aut(\N) \ltimes \N_R$. Of course, we also have $\N_L \leq \Aff(\N)$  via the embedding $t \mapsto ( C_{t^{-1}}, t)$, where $C_x$ denotes conjugation by $x$.

Since by assumption a twisted principal bundle has structure group contained in $\Aff(\N)$, it is clear that each fibre $M_b$ has a well-defined affine connection $\nabla^{\vca,b}$ which makes it affine-diffeomorphic to $(\N, \nabla^{\rm R})$.

\begin{definition}
An \emph{isomorphism of twisted principal $\N$-bundles} is a bundle map  which acts on the fibres by affine transformations. 
\end{definition}

If $f: M \to M'$ is an isomorphism of twisted principal bundles covering $\check f : B \to B'$, then it can be written in appropriate trivialisations $U_i \times \N$, $U'_j \times \N$ as $f(u,n)= (\check f(u), \varphi_u(n))$ for some $\varphi : U_i \to \Aut(\N)$.

\subsection{The adjoint bundle and connections on twisted principal bundles}

Recall that the adjoint bundle of a principal $\N$-bundle is the associated bundle for the adjoint representation \(\Ad:\sans{N}\to \GL(\mathfrak{n})\), obtained by composing the transition functions with $\Ad$. 
Associated to any twisted principal bundle \(\pi:M\to B\) there is an analogous construction: a Lie algebra bundle $\ad M \to B$, also called the adjoint bundle.

To define it, given a bundle atlas $\{U_i\times \N\}$ for  $\pi : M \to B$, we let $U_i\times \ngo$ be the vector bundle over $U_i$ whose sections are the vertical right-invariant fields in $U_i\times \N$. Composing the transition functions of $\{U_i\times \N\}$ with the map 
 % is the associated bundle to the $(\Aut)$The group \(\Aut(\sans{N})\ltimes \sans{N}\) acts on \(\mathfrak{n}\) via 
\begin{equation}\label{Eq:LieAlgaction}
	\Aut(\sans{N}) \ltimes \sans{N} \to \Aut(\mathfrak{n}),  \qquad (\varphi , n) \mapsto d \varphi_e \circ \Ad(n)
\end{equation}
we obtain the data to define a vector bundle $\ad M$ over $B$. Since the structure group is contained in $\Aut(\ngo)$, the fibres of $\ad M$ have a well-defined Lie algebra structure. (Equivalently, one may define $\ad M$ as the vector bundle associated to the $(\Aut(\N)\ltimes \N)$-principal bundle $\mathcal{P}\to B$ via the representation  \eqref{Eq:LieAlgaction}.)

Let $\vca := \ker d \pi$ be the vertical distribution on $M$. The key property of the adjoint bundle, which holds by construction, is that local sections of $\ad M$ are in one-to-one correspondence with invariant sections of $\vca$, i.e.~vertical invariant vector fields on $M$.  More generally, the same is true for  sections of tensor products of $\ad M$ and the corresponding invariant tensor fields on $M$. 

Notice also that an isomorphism of twisted principal bundles $f : M_1 \to M_2$ maps vertical invariant fields onto vertical invariant fields, thus inducing a vector bundle map 
\[
    f_\ad : \ad M_1 \to \ad M_2
\]
that preserves the Lie algebra structure of the fibres.

\smallskip

There is a straightforward generalisation of the notion of principal connection to this setting:

\begin{definition}\label{def_connection}
A \emph{connection} on a twisted principal bundle \(\pi:M\to B\) is an invariant distribution \(\mathcal{H} \subset TM\) complementary to \(\mathcal{V} = \ker d\pi\).
\end{definition}

Precisely as in the principal bundle setting, one may equivalently view a connection as the kernel of an invariant, $\vca$-valued one-form on $M$. Its restriction to each trivialisation $U_i\times \N$ is an $\ngo$-valued one-form satisfying the usual properties of a principal connection.
In this article, connections will often arise as orthogonal complements of \(\mathcal{V}\subset TM \) with respect to some invariant metric.
% (see Definition~\ref{def:InvariantMetrics} below).

An important feature of a connection $\hca \subset TM$ is its holonomy. Given $b\in B$ and a loop $\gamma : [0,1]  \to B$ with $\gamma(0) = \gamma(1) = b$, consider for each $p\in M_b$ the unique $\hca$-horizontal lift $\hat \gamma_p : [0,1] \to M$ with $\hat \gamma_p(0) = p$. 
% The map $M_b \to M_b$ sending $p$ to $\hat \gamma_p(1)$ is the holonomy of $\hca$ with respect to the loop $\gamma$. 
\begin{definition}\label{def_holo}
    The \emph{holonomy} of the connection $\hca$ at $b\in B$ with respect to the loop $\gamma$ is the map
    \[
        M_b \to M_b, \qquad p \mapsto \hat \gamma_p(1).
    \]
\end{definition}
Considering all possible loops based at $b$, and fixing an identification $M_b\simeq \N$ coming from a local trivialisation, we obtain the so-called \emph{holonomy group}, which is naturally a subgroup of the structure group (which is itself a subgroup of $\rho(\pi_1(B)) \ltimes \N$). It is well-defined up to conjugation.

\smallskip

For most of the article (except for \Cref{Sec:ExpandingSolitons}), we will be interested in \emph{flat connections}: the case where \(\mathcal{H}\) is an integrable distribution. Let $\hca$ be a flat connection on $M$, and fix a fibre $M_b$ and an identification $M_b\simeq \N$. By flatness, the horizontal lift of piecewise-smooth loops in $B$ based at $b$ is independent of the homotopy class of the loop, and thus gives rise to a well-defined \emph{monodromy representation}
\begin{equation}\label{eqn:rhohca}
    \rho_\hca : \pi_1(B) \to \rho(\pi_1(B))\ltimes \N.
\end{equation}
Clearly, the image of $\rho_\hca$ must contain $\rho(\pi_1(B))$, but there may be also a translational part. We will often be dealing  with flat connections whose monodromy $\rho_\hca$ equals the twisting homomorphism $\rho$.
% point $p\in M$, giving an indentification of the fibre

\begin{example}\label{ex_rho_hca}
Let $\rho_\hca \in \Hom(\pi_1(B), \Aut(\N)_\delta \ltimes \N)$ be a homomorphism and set $\tilde M := \tilde B \times \N$. The diagonal right action of $\pi_1(B)$ on $\tilde M$ given by \((b,x)\cdot \gamma = (b \cdot \gamma,\rho_\hca(\gamma^{-1})x)\) is properly discontinuous and $M:= \tilde M / \pi_1(B)$ is a twisted principal $\N$-bundle over $B$, with twisting homomorphism $\rho := q \circ \rho_\hca$, where $q : \Aut(\N)\ltimes \N \to \Aut(\N)$ is the `quotient by $\N$' projection. The trivial connection $T \tilde B$ on $\tilde M$ induces a connection $\hca$ on $M$ with monodromy $\rho_\hca$. It is not hard to see that in fact we have 
\[
    M = \tilde B \times_{\rho_{\hca}} \N.
\]
\end{example}

We note that flat connections always exist when $\N$ is contractible (in particular, for $\N$ simply-connected and nilpotent):

\begin{prop}\label{prop_structure_rho}
Let \(\sans{N}\) be a contractible Lie group and let \(M\) be a twisted principal \(\sans{N}\)--bundle with twisting homomorphism \(\rho:\pi_1(B)\to \Aut(\sans{N})\). Then, 
there exists a flat connection $\hca$ on $M$ with monodromy $\rho$. Moreover, given such a connection, there is an isomorphism of flat twisted principal bundles $M \simeq \tilde B \times_\rho \N.$
\begin{proof}
Let $\pca \to B$ be the principal $(\rho(\pi_1(B))_\delta\ltimes \N)$-bundle to which $M$ is associated. Since $(\rho(\pi_1(B))\ltimes \N) / \rho(\pi_1(B)) \simeq \N$ is contractible by assumption, there is a reduction of the structure group of $\pca$ to $\rho(\pi_1(B))_\delta$ \cite{KobayashiNomizu63}*{Ch.~I, Prop.~5.6 $\&$ Thm.~5.7}. Consider a bundle atlas for $\pca$ associated to a cover $\{U_i\}$ of $B$, with locally constant transition functions taking values in $\rho(\pi_1(B))$.  As $M$ is the associated bundle, a bundle atlas for it is given by $\{U_i \times \N\}$, with the same transition functions. We notice that this data is also the data defining the associated bundle to the principal $\pi_1(B)$-bundle $\tilde B \to B$ with respect to $\rho$. Hence, $M \simeq \tilde B \times_\rho \N$. 
\end{proof}
\end{prop}

\begin{remark}
Comparing \Cref{ex_rho_hca} and \Cref{prop_structure_rho}, we see that $\tilde B \times_{\rho_{\hca}} \N$ and $\tilde B \times_\rho \N$ are isomorphic as twisted principal bundles; however, this isomorphism need not preserve the flat connections (the monodromy $\rho_\hca$ may have a non-trivial translational part).
\end{remark}

If \(M\simeq\tilde{B}\times_\rho \sans{N}\) is as in  \Cref{prop_structure_rho}, the adjoint bundle satisfies $	\ad M \simeq \tilde{B}\times_{\rho} \mathfrak{n}$. It
inherits a flat  connection whose monodromy is still \(\rho\) (under the identification \(\Aut(\sans{N})\simeq \Aut(\mathfrak{n})\), assuming that $\N$ is simply-connected). 
From the principal bundle situation one can deduce the following important correspondence:

\begin{prop}\label{prop_sectionsadM}
There is a one-to-one correspondence between:
\begin{enumerate}
	\item[(i)] Invariant vertical vector fields on $ M = \tilde B \times_\rho \N$; 
	\item[(ii)] Sections of $\ad M = \tilde B \times_\rho \ngo$;
    \item[(iii)] $\rho$-equivariant maps $s : \tilde B \to \mathfrak{n}$:
	\begin{equation}\label{eqn_Btilde_n}
		s(b\cdot \gamma) = \rho(\gamma^{-1})\cdot s(b),\qquad \gamma\in \pi_1(B), \quad b\in B.
	\end{equation}
\end{enumerate}
\end{prop}

\begin{proof}
The equivalence between (i) and (ii) follows by the construction of $\ad M$, as mentioned above. The equivalence between (ii) and (iii) is a standard result on sections of a bundle (here $\ad M$) associated to a principal bundle (here, the universal cover $\tilde B\to B$), see e.g.~\cite{Michor_book}*{18.12}. 
\end{proof}

Similarly, we have the analogous result for arbitrary tensor fields; we state it in the case of symmetric $(0,2)$-tensors:

\begin{prop}\label{prop_metricsadM}
There is a one-to-one correspondence between:
\begin{enumerate}
	\item[(i)] Invariant, symmetric $(0,2)$-tensors on the vertical distribution $\vca$ of $ M = \tilde B \times_\rho \N$; 
	\item[(ii)] Sections of $\Sym^2(\ad M)$;
    % metrics on $\ad M = \tilde B \times_\rho \ngo$;
    \item[(iii)] $\rho$-equivariant maps $s : \tilde B \to \Sym^2(\ngo^*)$: 
	\begin{equation}
		s(b\cdot \gamma) = \rho(\gamma^{-1})\cdot s(b),\qquad \gamma\in \pi_1(B), \quad b\in B.
	\end{equation}
    The action of $\pi_1(B)$ on $\Sym^2(\ngo^*)$ is obtained from composing $\rho : \pi_1(B) \to \Aut(\ngo) \subset \GL(\ngo)$ with the natural extension of the standard $\GL(\ngo)$-action on $\ngo$ to $\Sym^2(\ngo^*)$ \eqref{eqn:GLnacts_sym2}.
\end{enumerate}
\end{prop}

\subsection{Invariant metrics}

Let \(\pi:M\to B\) be a twisted principal $\sans{N}$-bundle. Recall that there is a well-defined notion of invariance for tensor fields on $M$ (\Cref{def_invariant}). In particular, we may consider invariant Riemannian metrics on $M$. These have a simple description in terms of the pullback bundle $\tilde M \to \tilde B$, which is both a principal $\N$-bundle over $\tilde B$, and a principal $(\pi_1(B)\ltimes_\rho \N)$-bundle over $B$:

\begin{prop}
A metric \(g\) on \(M\) is invariant if and only if its pullback $\tilde g$ to \(\tilde{M}\) is \(\pi_1(B)\ltimes_\rho \sans{N}\)-invariant in the usual sense. 
\end{prop}
\begin{proof}
Since \(\tilde{g}\) is the lift of \(g\), and \(\pi_1(B)\) is the group of deck transformations of the covering $\tilde M \to M$, it is clear that \(\tilde{g}\) is \(\pi_1(B)\)-invariant, regardless of the symmetries of $g$. On the other hand, let $\{U_i\}$ be a bundle atlas of $M$ with constant transition functions into $\Aut(\sans{N})$, and let $\{\hat U_i\}$ be the corresponding bundle atlas for $\tilde M$. Since the covering map \(\hat U_i \to U_i\) is $\sans{N}$-equivariant, the $\sans{N}$-action on $\hat U_i$ is $\tilde g$-isometric if and only if the $\sans{N}$-action on $U_i$ is $g$-isometric, and the proposition follows. 
\end{proof}

Given an invariant metric $g$ on $M$, there is a unique metric $g_B$ on $B$ such that $\pi : M \to B$ is a Riemannian submersion. We also get a connection $\hca := \vca^\perp$, and an invariant bundle metric $g|_{\vca\times \vca}$ on the vertical distribution.  The holonomy of $g$ is by definition the holonomy of the connection $\hca$ (\Cref{def_holo}).

In this context, \Cref{prop_metricsadM} yields:

\begin{prop}\label{prop_inv_metrics}
An invariant metric $g$ on  $M$ is determined by the data $(g_B, \hca, h)$, where:
\begin{enumerate}
	\item[(i)] $g_B$ is a Riemannian metric on the base \(B\),
	\item[(ii)]  $\mathcal{H} = (\ker d \pi)^\perp$ is a connection,  and
	\item[(iii)] $h$ is a bundle metric on $\ad M$.
\end{enumerate}
Moreover, if $\hca$ is flat with monodromy $\rho$, then it is determined by a flat affine connection $\nabla^\ad$ on $\ad M$, and (iii) may be replaced by
\begin{enumerate} 
	\item[(iii)'] $h$ is a $\rho$-equivariant map $h : \tilde B \to \sym^2_+(\mathfrak{n}^*)$.
\end{enumerate}
\end{prop}

When writing that an invariant metric $g$ is determined by the data $(g_B, \nabla^\ad, h)$ for a flat connection $\nabla^\ad$, we will implicitly assume that the monodromy of $\hca$ is $\rho$.

The following may be considered as a parametrised version of Wilson's theorem on the isometries of nilmanifolds \cite{Wil82}:

\begin{prop}\label{prop:param_wilson}
Let \(\pi_i:M_i\to B_i\), \(i=1,2\), be twisted principal $\N$-bundles,  $\N$  nilpotent and simply-connected, with invariant metrics \(g_i\) determined by $(g_{B_i}, \nabla^{\ad,i}, h_i)$, where $\nabla^{\ad, i}$ are flat connections on $\ad M_i$. Let 
\(f:(M_1,g_1)\to (M_2,g_2)\) be a fibre-preserving isometry covering a diffeomorphism $\check f : B_1 \to B_2$. Then, $f$ is a twisted principal bundle isomorphism, and:
\begin{enumerate}
\item[(i)] \(\check{f}:(B_1,g_{B_1})\to (B_2,g_{B_2})\) is an isometry;
\item[(ii)] \({f}_\ad: \ad M_1\to \ad M_2\) covers \(\check{f}\) and satisfies \begin{equation}\label{eqn_BMP}
 h_1= {f}_\ad^*h_2,\qquad  \nabla^{{\rm ad}, 1}={f}_\ad^*\nabla^{{\rm ad}, 2}.
\end{equation}
\end{enumerate}
Conversely, a Lie algebra bundle isomorphism \({f}_\ad:\ad M_1\to \ad M_2\) satisfying (i) and (ii) gives rise to a 
fibre-preserving isometry \(f:(M_1,g_1)\to (M_2,g_2)\).
\begin{proof}
Since \(f\) is a fibre-preserving isometry, we have  
\(f_*(\mathcal{V}_1)=\mathcal{V}_2\) and $f_* (\hca_1) = \hca_2$. Also, $f$ covers a diffeomorphism \(\check{f}:B_1\to B_2\), which is isometric since for \(p\in \pi_1^{-1}(b)\subset M_1\), its differential is a composition of local isometries: \[d\check{f}_b = d\pi_2\circ df_p\circ d\pi_1|_{\mathcal{H}_1}^{-1}.\] 

Given \(b\in B_1\), the restricted map \( f|_{M_b}:M_{1,b}\to M_{2,\check{f}(b)}\) is an isometry between  fibres. Since the fibres admit simply-transitive isometric actions by a nilpotent Lie group, \cite{Wil82} implies that 
\[
    f|_{M_b} : (M_{1,b}, \nabla^{\vca_1, b}) \to (M_{2, \check f(b)}, \nabla^{\vca_2, \check f(b)})
\]
is an affine diffeomorphism with respect to the canonical vertical connections (\Cref{sec:can_vert}). Hence, $f$ is an isomorphism of twisted principal $\N$-bundles. It therefore preserves vertical invariant fields, thus giving rise to a vector bundle isomorphism $f_\ad : \ad M_1 \to \ad M_2$ which is of course a fibrewise isometry with respect to the bundle metrics $h_1, h_2$.  From $f_* (\hca_1) = \hca_2$  we deduce that $f_\ad^* \nabla^{\ad, 2} = \nabla^{\ad, 1}$.

Conversely, suppose that we are given \(f_\ad\) and \(\check{f}\) satisfying (i) and (ii). By replacing \(M_2\) with \(\check{f}^*M_2\), we may assume that \(f_\ad\) covers the identity. Hence, we can consider \(f_\ad\) as a section of \((\ad M_1)^*\otimes(\ad M_2)\). By \eqref{eqn_BMP}, \(f_\ad\) is parallel with respect to the induced flat connection: hence, we can identify \(f_\ad \) with an element of \(\Aut(\mathfrak{n})\) such that \(f_\ad \circ\rho_{1,*}(\gamma) =\rho_{2,*}(\gamma)\circ f_\ad\). Since \(\sans{N}\) is simply connected, we can integrate \(f_\ad\) to \(f\in \Aut(\sans{N})\) which satisfies \(f\circ\rho_1(\gamma)=\rho_2(\gamma)\circ f\). Then, recalling the identification \(\tilde M_i \simeq \tilde B\times \sans{N}\) given by the flat connections, the map \[ \tilde B\times \sans{N}\to  \tilde B\times \sans{N} ,\qquad (b,x)\mapsto (b,f(x)),\]
descends to a map \(f:M_1\to M_2\). The map \(f\) is isometric since \(f_*(\hca_1)=\hca_2\) and \(f_*(\vca_1)=\vca_2\), it covers an isometry, and \(f_{\ad}^*h_2=h_1\).
\end{proof}
\end{prop}

\subsection{Ricci curvature of invariant metrics}

In this section, we record (without proofs) the formulas for the Ricci curvature of an invariant metric on a twisted principal bundle, which follow from O'Neill's formulas for the curvature of a Riemannian submersion. The presentation is based on \cite{alek_sol} (see also  \cite{Besse87}*{Chpt.~9} or \cite{NT18}{Prop.\ 8.1},), but contains some additional remarks that are new to this paper.

Throughout this section, \(U\) will denote a local section of $\ad M$, $\hat U $ the corresponding invariant vertical vector field on $M$. Since the fibres of $\ad M$ have a natural Lie algebra structure, $\ad U := [U,\cdot] \in \Der(\ad M)$ will denote, as usual, the (fibrewise) adjoint representation.
We will also write \(X, Y\) for vector fields on $B$, and $\bar X, \bar Y$ for their corresponding horizontal lifts to $M$. 

An invariant metric $g$ on $M$ is determined by a triple $(g_B, \hca, h)$ as per \Cref{prop_inv_metrics}, and $\pi : M \to B$ is a Riemannian submersion. The curvature of $\nabla^\ad$  
% where $g_B$ is defined so that $\pi : M \to B$ is a Riemannian submersion. The connection $\hca$ induces an affine connection $\nabla^\ad$ on $\ad M$, whose curvature 
is O'Neill's $A$-tensor, which we view as an $\ad M$-valued $2$-form on $B$: for $X,Y \in TB$,  $A_{X,Y} \in \Gamma(\ad M)$ corresponds to the invariant vertical field $\tfrac12 \vca [\bar X, \bar Y]$. Naturally,  $A=0$ if and only if $\hca$ is an integrable distribution. For each $U \in \Gamma (\ad M)$, we also set $AU$ to be the ($\R$-valued) $2$-form on $B$ given by $(X,Y) \mapsto h(A_{X,Y} , U)$. 

For each $X\in TB$, we denote by $L_X$ the shape operator of the fibres in $M$, viewed as a $h$-self-adjoint endomorphism of $\ad M$. It is implicitly defined by
\begin{equation}\label{eqn:shape}
    h(L_X U, V) = g(\nabla_{\bar X} \hat U, \hat V).
\end{equation}
The fibres' mean curvature vector $\bar \mcv$ is the basic horizontal vector field given by  the trace of the second fundamental form after raising an index. The corresponding vector field $\mcv$ on $B$ satisfies $g_B(\mcv, X) = - \tr L_X$, and  is thus given in terms of any orthonormal frame  $\{X_j\}$ for $B$ by
\[
    \mcv = - \sum_{j=1}^{\dim B} (\tr L_{X_j}) \cdot X_j.
\]
% where $\{X_j\}$ is orthonormal.

The fibre's Ricci curvature is an invariant section of $\Sym^2(\vca^*)$, and it thus corresponds to a section $\Ric_h$ of $\Sym^2(\ad M)^*$. 

Following \cite{Cor88}, we also make the following definition:
\begin{definition}\label{def_tensionfield}
Let \(\nabla^{\rm ad}\), $h$ be a connection and bundle metric on \(\ad M\), and let \(g_B\) be a metric on \(B\). The \emph{tension field} of \(h\) is the section  \( \tau(h) \in \Gamma \sym^2(\ad M)^*\) defined by 
\begin{equation}\label{eqn:deftau}
\tau(h) = -2h(\nabla^{+,*}\theta\cdot,\cdot),
\end{equation}
where \(\nabla^{\rm ad}=\nabla^+ +\theta\) is the (unique) decomposition of \(\nabla^{\rm ad}\) into a  connection \(\nabla^+\) preserving $h$, and a \(h\)-self adjoint \(\End(\ad M)\)-valued one-form \(\theta\) (cf.~\Cref{Lemma:Tensionfield}). 
\end{definition}
In Definition~\ref{def_tensionfield}, we use the Levi-Civita connection to extend the covariant derivative \(\nabla^+\) to act on \(\End(\ad M)\)-valued tensors. Moreover, $\nabla^{+,*}$ denotes the formal adjoint of \(\nabla^+\) with respect to the \(L^2\)-inner product on \(\End(\ad M)\)-valued tensors induced by \(h\) and \(g_B\). On \(\End(\ad M)\)-valued one-forms, it is given by
\begin{equation}\label{eqn:nabla+*theta}
\nabla^{+,*}\theta = -\sum_{i=1}^{\dim B} (\nabla^+_{\bar{X}_i}\theta)_{\bar{X}_i} . 
\end{equation}

The Ricci curvature of \((M,g)\) can then be expressed in terms of the Ricci curvature of \((B,g_B)\), the Ricci curvature of the fibres (with the induced metric), and the tensors \(A\) and \(T\); see \cite{Besse87} or \cite{alek_sol}*{Theorem~2.2}.

\begin{prop}\label{Prop:SubmersionRicci}
If $g$ is an invariant metric on a twisted principal $\N$-bundle  \(\pi:M^n\to B^d\), with $\N$ unimodular, then,
% determined by the triple \((g_B,\mathcal{H}, h)\). Let \(U\) be a \(\sans{N}\)-invariant vertical vector field, \(X\) an \(\sans{N}\)-invariant horizontal vector field. Let \(\bar{U}\) and \(\bar{X}\) be the corresponding section of \(\ad M\) and \(TB\). 
\begin{subequations}\label{Eq:SubmersionCurvature}
\begin{align}
\label{Eq:SubmersionCurvature(V)}\Ric_g(\hat U, \hat U)&= \Ric_h ({U},{U}) + | AU |^2-\tfrac{1}{2}\tau(h)({U},{U}) + h(L_{{\mathrm H}} U, U) ,\\
\label{Eq:generalSubmersionCurvature(B)}\Ric_g(\bar X, \bar X) &= \Ric_{g_B}({X}, X) - |L_X|^2 - 2 \,|A_X|^2 +g_B(\nabla_{ X}{{\mathrm H}}, X).
\end{align}
\end{subequations}
% Here, \(\tau(h)\) is the \emph{tension field} of \(h\) and \(\mathrm H = \mathcal{H}\sum_{i=1}^n \nabla_{U_i}U_i\) is the mean curvature vector field. 
Furthermore, when $A=0$, we view $h : \tilde B \to \Sym^2_+(\ngo^*)$, and it holds that
\begin{subequations}\label{Eq:SubmersionCurvature_A=0}
\begin{align}
    \label{Eq:SubmersionCurvature(M)} \Ric_g(\hat U, \bar X) &= \tr \big( \ad {U} L_{{X}} \big), \\
    \label{Eq:SubmersionCurvature(B)} \Ric_g(\bar X, \bar X) &= \Ric_{g_B}({X}, X) - h^*g_{\mathsf{sym}}( \tilde X,  \tilde X)  +g_B(\nabla_{ X}{{\mathrm H}}, X).
\end{align}
\end{subequations}
Here $\tilde X$ denotes the lift of $X$ to $\tilde B$, and $g_{\rm sym}$ is the symmetric metric on $\Sym^2_+(\ngo^*)$ \eqref{eqn:g_sym}.
\begin{proof}
Locally, the geometry is that of a principal bundle with an invariant metric, thus \eqref{Eq:SubmersionCurvature} follows directly from \cite{alek_sol}*{Thm.~2.2} and \Cref{Lemma:Tensionfield} below.  Regarding \eqref{Eq:SubmersionCurvature_A=0}, \eqref{Eq:SubmersionCurvature(M)} is precisely \cite{alek_sol}*{Prop.~3.6} (since $\N$ is unimodular). Finally, when $\hca$ is flat we may view $h : \tilde B \to \Sym^2_+(\ngo^*)$ thanks to \Cref{prop_inv_metrics}, and by \cite{alek_sol}*{Lemmas 4.1 $\&$ 5.4} we have 
\[|L_X|^2= \tr L_X^2 =g_{\rm sym}(dh  \tilde X, dh\tilde X)=h^*g_{\rm sym}(\tilde X, \tilde X).\] 
\end{proof}
\end{prop}

The Levi-Civita connection \(\nabla^g\) of \(g\) defines a connection \(\nabla^{\rm LC}\) on \(\ad M\) in the following way: \(\nabla^{\rm LC}_X U\) is the section of \(\ad M\) that corresponds to the invariant vertical vector field \({\mathcal V}\nabla^g_{\bar{X}}\check{U}\). Notice that \(\nabla^gg=0\) implies \(\nabla^{\rm LC}h=0\). Recall also that we view \(L\) as an \(\End(\ad M)\)-valued one-form on \(B\). The following is a generalisation of \cite{Lot07}*{Proposition~4.17}.

\begin{lemma}\label{Lemma:Tensionfield}
The decomposition of $\nabla^\ad$ into a connection preserving $h$ and a \(h\)-self-adjoint \(\End(\ad M)\)-valued one-form is
\begin{equation}\label{eqn_adDecomp}
    \nabla^{\rm ad} = \nabla^{\rm LC} - L,
\end{equation}
where $\nabla^{\rm LC}$ is induced by the Levi-Civita connection on $M$. 
Consequently, the tension field of \(h\) is given by 
\begin{equation}\label{eqn:tau=2divL}
\tau(h)(U,U) = 2\sum_{i=1}^n \m{(\nabla^g_{\bar{X}_i}L)_{\bar{X}_i}\hat{U}}{\hat{U}}.
\end{equation} 
If \(\nabla^{\rm ad}\) is flat, then \(\tau(h)\) is the tension field (in the sense of harmonic maps) of the \(\rho\)-equivariant map \(h:(\tilde{B},\tilde{g}_B)\to (\sym_+^2(\mathfrak{n}^*),g_{\sans{sym}})\).
\begin{proof}
Recall \cite{NT18}*{Lem.\ 8.1} that \(\nabla^{\rm ad}\) is determined by the relation 
\[
% lala
    \big( \nabla^\ad_X U \big)^\wedge = [\bar{X},\hat{U}].
\]
Using that \(\nabla^g\) is torsion free, and that the bracket of vertical and basic vector fields is is vertical, we have 
\[
    [\bar{X},\hat{U}] = \mathcal{V}\nabla^g_{\bar{X}}\hat{U} - \mathcal{V}\nabla^g_{\hat{U}} \bar{X}. 
\]
By definition, the section of \(\ad M\) corresponding to the first term on the right-hand-side is \(\nabla^{\rm LC}_X U\). Moreover, by \eqref{eqn:shape}, the remaining term corresponds precisely to $-L_X U$. Equation \eqref{eqn:tau=2divL} now follows from \eqref{eqn:deftau},  \eqref{eqn:nabla+*theta}. 
\end{proof}
\end{lemma}

\section{The twisted harmonic-Einstein equations}\label{Sec:Ansatz}

As mentioned before, our original motivation is to produce new examples of invariant expanding Ricci solitons on twisted principal $\N$-bundles, where $\N$ is a simply-connected nilpotent Lie group. In this section we introduce an Ansatz for that purpose, which reduces the soliton equation \eqref{eqn_eRS} to the so-called twisted harmonic-Einstein equations \eqref{eqn_twistharmEin}.  In subsequent subsections we  discuss the solutions to \eqref{eqn_eRS} in particular cases.  Further, in \Cref{Sec:ExpandingSolitons}, motivated by the asymptotic limits of collapsing Ricci flows \cite{LP25}, we will introduce the notion of affine expanding soliton, and show that they all satisfy our Ansatz. 
% all expanding Ricci solitons arising as limits of Ricci flows on closed manifolds as in \Cref{thm_LP} indeed satisfy the Ansatz.  

Let $\pi:M\to B$ be a twisted principal \(\sans{N}\)--bundle with $\N$ nilpotent and simply-connected, with twisting homomorphism \(\rho\in \Hom(\pi_1(B),\Aut(\sans{N}))\).  Fix a Haar measure $\nu$ on $\N$, let $|\nu|$ be the corresponding density, and denote by $\Aut_{\pm 1}(\N)$ the subgroup of automorphisms of $\N$ preserving $|\nu|$. Recall that invariant metrics on $M$ are determined by the data $(g_B, \hca, h)$, where $g_B$ is a metric on $B$, $\hca$ is a connection on $M$, and $h$ is a bundle metric on $\ad M$ (\Cref{prop_inv_metrics}). We call $\nabla^\ad$ the affine connection on $\ad M$ determined by $\hca$.

\begin{ansatz}\label{Ansatz} For a pair $(g,X)$ of an invariant metric $g$ corresponding to $(g_B, \hca, h)$ and a (not necessarily invariant) vector field $X$ on $M$,  we assume:
\begin{enumerate}[label=(\roman*)]
    \item $\hca$ is flat and its monodromy is $\rho : \pi_1(B) \to \Aut_{\pm 1}(\N)$; \label{item:A_flat}
    \item The fibres of $\pi : M \to B$ are minimal submanifolds and Ricci solitons with soliton constant $-1/2$, and  $\nabla^\ad (\RicE{h}) = 0$; \label{item:A_soliton}
    \item  $X$ is vertical and it restricts to a soliton vector field of the fibres. \label{item:A_X}
\end{enumerate}
\end{ansatz}

Some comments are in order. Firstly, thanks to \Cref{prop_structure_rho}, flatness in  \ref{item:A_flat} yields an isomorphism $M\simeq \tilde B \times_\rho \N$,
and by \Cref{prop_inv_metrics} the bundle metric becomes a $\rho$-equivariant map 
\[
    h : \tilde B \to \Sym^2_+(\ngo^*).
\]
The monodromy assumption has two parts: on one hand, there is no translational part (cf.~\eqref{eqn:rhohca}). On the other hand, the unimodularity restriction on $\rho$ further ensures the existence of a globally-defined, $\nabla^\ad$-parallel density $|\nu_\ad|$ on the fibres of $\ad M$.

Next, when lifted to $\tilde M \simeq \tilde B \times \N$, a metric $g$ satisfying \Cref{Ansatz} is of the form 
\[
	\tilde g = \tilde g_B + h,
\]
and by \ref{item:A_soliton}, the fibres $(N, h_{\tilde b})$ are all nilsolitons. The significance of the condition $\nabla^\ad (\RicE{h}) = 0$ is that the Ricci endomorphism of the fibres is independent of the $\tilde B$-coordinate. 
See \Cref{app_nil} for further details about nilsolitons.

Finally, in analogy with the homogeneous case (see \cite{solvsolitons}*{$\S$2} and the proof of \Cref{prop_nilsoliton}), the vector field $X$ can be obtained from the nilsoliton derivation field $D$:
% defined by $\Ric_h$ as in \eqref{eqn:h_soliton}:

\begin{lemma}\label{Lemma:SolVec}
Let $D\in \Gamma \Der(\ad M)$ be a the derivation field corresponding to an invariant endomorphism $D_\vca \in \End(\vca)^{\rm inv}$, and let $g$ be an invariant metric on $M$ defining a flat connection $\hca$ whose monodromy has no translational part.
% Assume that the invariant metric $g$ on $M$ defines a flat connection $\hca$ with holonomy $\rho$. 
Then, there exists a vertical vector field $X_D \in \Gamma(TM)$ so that 
\[
 \mathcal{L}_{X_D} g =  g(D_M \cdot, \cdot) + g(\cdot, D_M \cdot),
\]
where $D_M \in \End(TM)^{\rm inv}$ is given according to $TM = \vca \oplus \hca$ by 
\[
    D_M := \begin{pmatrix}
        D_\vca & 0 \\
        0 & 0
    \end{pmatrix}.
\]
% for any invariant metric $g$ on $M$.
\begin{proof}
The idea is to use local trivialisations, and argue as in \cite{solvsolitons}*{$\S$2} fibrewise. The only thing to be checked is that the locally-defined objects behave well with respect to the transition maps, so that they glue to define a global object. 

More precisely, consider a bundle atlas $\{ U_i \times \N \}_{i\in I}$ for $M$ adapted to $\hca$, with locally constant transition functions $\tau_{ij} : U_i \cap U_j \to \rho(\pi_1(B)) \subset \Aut(\N)$. Since $\N$ is simply-connected, there is a canonical Lie group isomorphism $\Aut(\N) \simeq \Aut(\ngo)$, $\varphi \mapsto \underline{\varphi}$, and this defines a bundle atlas $\{U_i\times \ngo \}_{i\in I}$ for $\ad M$, with transition functions $\underline{\tau}_{ij}$. In each trivialisation $U_i\times \ngo$, the derivation field $D$ becomes $D_i : U_i \to \Der(\ngo)$. For each $t$ this gives $\exp(t D_i) : U_i \to \Aut(\ngo)$, and under a change of trivialisation these change via
\[
    \exp(t D_i) = \underline{\tau}_{ij}  \circ \exp(t D_j)  \circ \underline{\tau}_{ij}^{-1}.
\]
Let $\phi_t^i \in \Aut(\N)$ be defined by $\underline{\phi}_t^i = \exp(t D_i)$. Clearly, these transform via
\[
    \phi_t^i = \tau_{ij} \circ \phi_t^j \circ \tau_{ij}^{-1},
\]
and thus, for each $t\in \R$, the locally-defined maps 
\[
    U_i \times \N \to U_i\times \N, \qquad (u,n) \mapsto (u,\phi_t^i(u))
\]
glue up to define a diffeomorphism $\phi_t$ of M. This gives rise to a one-parameter group $(\phi_t)_{t\in \R}$ of diffeomorphisms of $M$ that are also twisted principal bundle isomorphisms covering the identity (see \Cref{sec:can_vert}). They generate a vector field $X_D\in \Gamma(TM)$. Computing fibrewise in a local trivialisation as in \cite{solvsolitons}*{p.~3} we deduce the claimed formula for the Lie derivative (we notice that there is a sign error in the formula from \cite{solvsolitons}, which does not affect the results of that paper whatsoever). 
\end{proof}
\end{lemma}

\begin{remark}
When $\N \simeq \R^d$ is abelian, conditions \ref{item:A_soliton} and \ref{item:A_X} in \Cref{Ansatz} reduce to the minimality of the orbits and the fact that, on a trivialisation $U\times \R^d$ adapted to the flat connection $\hca$,  the soliton vector field is half the Euler vector field $x^i \tfrac{\partial}{ \partial x^i}$ on $\R^d$ . This is precisely the Ansatz used by Lott in \cite{Lot07}*{$\S$4}.
\end{remark}

The next result shows that, under Ansatz \ref{Ansatz}, the expanding Ricci soliton equation \eqref{eqn_eRS}  reduces to the twisted harmonic-Einstein equations \eqref{eqn_twistharmEin}. 

\begin{prop}\label{prop_AnsatthenHE}
Let $(M,g,X)$ be an invariant expanding Ricci soliton, with $g$ determined by the data $(g_B,\hca, h)$, and assume that \Cref{Ansatz} is satisfied. Then, $(g_B, \rho, h)$ solves the twisted harmonic-Einstein equations 
% \eqref{eqn_twistharmEin} 
\begin{subequations}\label{Eq:HarmonicEinsteinEquations}
\begin{align}
\label{Eq:HarmonicEinsteinEquations(V)} \tau_{\tilde g_B}(h) &= 0 \, , \\
\label{Eq:HarmonicEinsteinEquations(B)} \Ric(\tilde g_B) &= -\tfrac12 \tilde g_B+h^*g_{\mathsf{sym}},
\end{align}
\end{subequations}
for the pair $(B, \G/\K)$, where $\rho : \pi_1(B) \to \G$, $h : \tilde B \to \G/\K$, and the groups $\G$ and $\K$ are given by
\begin{equation}\label{eqn:G,K}
    \G = \{ \varphi \in \Aut(\ngo) : \varphi \beta \varphi^{-1} = \beta, \det \varphi = \pm 1 \}, \qquad \K = \G \cap \mathsf{O}(\ngo, \bar h).
\end{equation}
Here, $\beta  :=  2\, \RicE{\bar h}  \in \End(\ngo)$ is twice the Ricci endomorphism of a fixed nilsoliton metric $\bar h$ on $\N$ with cosmological constant $-\tfrac12$, and we use the canonical isomorphism $\Aut(\N) \simeq \Aut(\ngo)$.

Conversely, given a solution $(g_B, \rho, h)$ to \eqref{Eq:HarmonicEinsteinEquations}, there exists a $\rho$-twisted principal $\N$-bundle $M$ with flat connection $\hca$  such that the invariant metric $g$ determined by $(g_B, \hca, h)$ solves the expanding Ricci soliton equation \eqref{eqn_eRS} for some vector field $X$ on $M$, and the pair $(g,X)$ satisfies \Cref{Ansatz}.
\begin{proof}
Assume $(M,g,X)$ is a soliton satisfying \Cref{Ansatz}. Then, $A=0$ and $\mcv =0$ in \Cref{Prop:SubmersionRicci}.
Thus, from  \eqref{Eq:SubmersionCurvature(V)} we get $\tau_{\tilde g_B} (h) = 0$, since the Ricci curvature terms cancel out because the fibers are Ricci solitons with the same constant and the same vector field as $(M,g, X)$.
On the other hand, \eqref{Eq:SubmersionCurvature(B)} and
the soliton equation \eqref{eqn_eRS}   give 
\[
 \Ric_{\tilde g_B} =     -\tfrac12 \tilde g_B  -\tfrac12 \,  (\mathcal{L}_X g) |_{\mathcal H \times \mathcal{H}} +  h^* g_{\rm sym}.
\]
Since $X$ is vertical and $A=0$, we have $(\mathcal{L}_X g) |_{\mathcal H \times \mathcal{H}} = 0$,
thus \eqref{Eq:HarmonicEinsteinEquations(B)} follows.  By \ref{Ansatz}, \ref{item:A_soliton},  the image of $h$ consists of nilsoliton inner products on $\ngo$ with a fixed Ricci endomorphism $\beta$ and a fixed volume density (by minimality). \Cref{lem:solitons_connected} thus yields $h : \tilde B \to \G/\K$. 
Regarding $\rho_* : \pi_1(B) \to \Aut(\ngo)$, since it is the holonomy of $\nabla^\ad$, which makes the Ricci endomorphisms parallel,  its image must centralise $\beta$. Together with \ref{Ansatz}, \ref{item:A_flat}, it follows that $\rho_* : \pi_1(B) \to \G$.

\smallskip

Conversely, we must show that any solution $(g_B,\rho, h)$ to \eqref{Eq:HarmonicEinsteinEquations} gives rise to an invariant expanding Ricci soliton on $M$ satisfying \Cref{Ansatz}. We set $M := \tilde B \times_\rho \N$, which has a canonical flat connection $\hca$ induced from that of the universal cover $\tilde B \to B$, with monodromy $\rho$.
The data $(g_B, \hca, h)$ defines an invariant metric $g$ on $M$ satisfying \ref{item:A_flat}. 

We claim that also \ref{Ansatz},   \ref{item:A_soliton} holds. To see that, let $U_i \times \ngo$ be a trivialisation for $\ad M$ adapted to the flat connection $\nabla^\ad$.  The bundle metric becomes a map $h : U_i \to \G/\K$. The shape operators $L_X\in \mathfrak{g}$ are traceless because $\G \subset \SL(\ngo)$, thus $\mcv = 0$.    Also, since $\G \subset \Aut(\ngo)$, for each $b\in U_i$ there is an isomorphism of  metric Lie algebras $ \varphi(b) :(\ngo, \bar h) \to (\ngo, h_b)$, with $(\ngo, \bar h)$ a nilsoliton with constant $-\tfrac12$. The Ricci endomorphism is thus 
\[
    2 \RicE{h_b}= \varphi(b)^{-1} \beta \varphi(b) = \beta,
\]
independent of $b$. Invariantly, this means that it is $\nabla^\ad$-parallel.

We now prove  that $g$ satisfies \eqref{eqn_eRS} for a vector field $X$ satisfying \ref{item:A_X} in \Cref{Ansatz}. We first analyse \Cref{Eq:SubmersionCurvature(V)}. The fixed determinant condition in $\G$ yields $H = 0$, and together with flatness of $\nabla^\ad$ and \eqref{Eq:HarmonicEinsteinEquations(V)} we get that $\Ric|_{\mathcal{V} \times \mathcal{V}} = \Ric_h$ (under the correspondence from \Cref{prop_metricsadM}).   Since all inner products on the orbit $\G \cdot \bar h \subset \Sym^2_+(\ngo^*)$ define pairwise isometric, invariant Ricci soliton metrics on the fibres of $\pi : M \to B$, we have a derivation field 
\begin{equation}\label{eqn:D_from_Ric_h}
    D := - \RicE{h} - \tfrac12 h  \in \Gamma \Der(\ad M).
\end{equation}
\Cref{Lemma:SolVec} yields a vertical $X_D\in \Gamma(TM)$ such that 
\[
    \Ric_g|_{\mathcal{V} \times \mathcal{V}} = (-\tfrac12 g - \tfrac12 \mathcal{L}_{X_D} g)|_{\mathcal{V} \times \mathcal{V}}.
\]
Using \eqref{Eq:SubmersionCurvature(B)}, we see that the horizontal part of equation \eqref{eqn_eRS} is satisfied thanks to minimality of the orbits and \eqref{Eq:HarmonicEinsteinEquations(B)}.  Finally, we claim that the mixed terms in $\Ric_g$ vanish. To see that, given $b\in B$ we may identify the fibre $(\ad M)_b$  with a metric Lie algebra $(\ngo, h_0)$ whose Ricci endomorphism is $\tfrac12 \beta$. In view of \eqref{Eq:SubmersionCurvature(M)}, it suffices to show that every shape 
operator $L_X \in \End(\ngo)$ at $b$ satisfies the assumptions of \Cref{lem:D_perp_adU}. Viewing $h$ as a section of the flat $\G/\K$-bundle $\tilde B \times_\rho \G/\K$, we have
\[
   h^{-1} d h (X) = 2\, L_X.
\]
Thus, $L_X$ is both a $h$-self-adjoint endomorphism of $\ad M$, and tangent to the fibres of the flat $\G/\K$-bundle over $B$. By definition of $\G$, this implies that $L_X = D + D^T$ where $D\in \Gamma \Der(\ad M)$ commutes with $\beta$, and the proof is concluded. 
\end{proof}
\end{prop}

\begin{definition}\label{def:tHE_equivalence}
Two solutions \((\check{g}_1,\rho_1,h_1), (\check{g}_2,\rho_2,h_2)\) of the twisted harmonic-Einstein equations for the pair \((B,\sans{G}/\sans{K})\) are called \emph{equivalent} if there exists \(\check{f}\in \Diff(B)\) and \(\varphi \in \sans{G}\) such that \begin{equation}\label{eqn_tHEequiv}
\check{f}^*\check{g}_2 = \check{g}_1,\qquad  C_\varphi  \circ \rho_1 = \rho_2\circ \check{f}_*,\qquad  \varphi \cdot  h_1=h_2 \circ \check f,
\end{equation} where \(\check{f}_*:\pi_1(B)\to \pi_1(B)\) is the induced action on the fundamental group and \( C_\varphi(q) = \varphi q \varphi^{-1} \) for $q\in \G$. 
% f_\ad\big(\rho(\gamma)\big)f_\ad^{-1}. \)
\end{definition}

The significance of this equivalence notion relies on the following 

\begin{prop}\label{prop:equiv_iff_isom}
Let \((M_1,g_1)\) (resp.\ \((M_2,g_2)\)) be a Ricci soliton obtained from a solution \((\check{g}_1,\rho_1,h_1)\) (resp.\ \((\check{g}_2,\rho_2,h_2)\)) of the twisted harmonic-Einstein equation as in \Cref{prop_AnsatthenHE}. Then, there exists an isometric bundle map between \((M_1,g_1)\) and \((M_2,g_2)\) if and only if \((\check{g}_1,\rho_1,h_1)\) is equivalent to \((\check{g}_2,\rho_2,h_2)\).
\begin{proof}
By \Cref{prop:param_wilson}, an isometric bundle map between \((M_1,g_1)\) and \((M_2,g_2)\) is equivalently described by  an isometry \(\check{f}:(B_1,\check{g}_1)\to (B_2,\check{g}_2)\) and a Lie algebra bundle isomorphism \(f_\ad:\ad M_1\to \ad M_2\) such that \[h_1=f_\ad^*h_2, \qquad \nabla^{\ad,1}=f_\ad^*\nabla^{\ad,2}.\] 
After choosing trivialisations we obtain Lie algebra isomorphisms $(\ad M_1)_b \simeq \ngo \simeq (\ad M_2)_{\check f (b)}$ for some given point $b\in B$. Using parallel transport, the map $f_\ad$ is in turn completely determined by what it does at the fibre over $b$, and this action is simply by an automorphism $\varphi \in \Aut(\ngo)$.  The condition $f_\ad^* h_2 = h_1$ translates into $\varphi \cdot h_1 = h_2$. The fibre isomorphisms with $\ngo$ also yield explicit holonomy representations $\rho_1 : \pi_1(B,b) \to \Aut(\ngo)$ and $\rho_2 : \pi_1(B, \check f(b)) \to \Aut(\ngo)$, and it is easy to check that $\varphi$ must satisfy the compatibility condition $\varphi \rho_1([\gamma]) \varphi^{-1} = \rho_2([\check f \circ \gamma])$ for all $[\gamma] \in \pi_1(B,b)$.

Suppose now that two triples are equivalent as per \ref{eqn_tHEequiv}. Then, the above observations  imply that the corresponding solitons are isometric via a bundle map.

Conversely, suppose that there is an isometric bundle map determined by $\check f$ and $\varphi$ as above. Since $f_\ad$ is fibrewise an isometry of metric Lie algebras,  it intertwines the fibres' Ricci endomorphisms, thus $\varphi$ commutes with $\beta$. It also preserves the parallel density $\nu_\ad$, thus $\det \varphi= \pm 1$. 
\end{proof}
\end{prop}

We will now describe three special cases of the harmonic-Einstein equations on twisted principal bundles.

\subsection{Invariant expanding Ricci solitons on principal bundles}\label{sec:untwisted}

To justify the need to deal with twisted principal bundles, we now show that, if $\rho$ is trivial, then \Cref{Ansatz} only yields Riemannian products (cf.~the Splitting Conjecture for negative Einstein manifolds with symmetry \cite{BL23b}):

\begin{prop}\label{Prop:UntwistedSolitons}
Let \((M, g, X)\) be an expanding Ricci soliton obtained from a solution $(g_B, \rho, h)$ to the twisted harmonic-Einstein equations \eqref{Eq:HarmonicEinsteinEquations} with trivial $\rho$. 
% with a free, proper, isometric action by \(\sans{N}\) such that \(B=M/\sans{N} \) is compact, the horizontal distribution is integrable, and the fibres are pairwise isometric to nilsolitons. 
Then, \((M,g)\) is isometric to the Riemannian product \((B,g_B)\times (\sans{N},\bar{h})\) of a compact Einstein manifold and a nilsoliton.
\begin{proof}
% Since the bundle $M \to M/\N = B$ is principal, the monodromy $\rho_\hca$ of the flat connection $\hca$ determined by $g$ takes values in $\N_L$. 
Since the action of $\pi_1(B)$ on $\ngo$ is trivial, 
% If \(\tau(\pi_1(B))\le \sans{Z}(\sans{N})\), then 
\(h:\tilde{B}\to \sym_+^2(\mathfrak{n}^*)\) is \(\pi_1(B)\)--invariant. 
% Indeed, \[h_{b\cdot \gamma} = \Ad_{\tau(\gamma)}^{-1}\cdot h_b = h_b,\] since \(\tau(\gamma) \in \sans{Z}(\sans{N})=\ker \Ad\). 
Hence, \(h\) descends to a harmonic map \(\check{h}:(B,g_B)\to (\sym_+^2(\mathfrak{n}^*),g_{\mathsf{sym}})\), which must necessarily be constant \cite{EellsLemaire78}*{Paragraph~4.5}. It follows that $M \simeq B \times \N$ is a Riemannian product. Moreover, from \eqref{Eq:HarmonicEinsteinEquations(B)}, $\tilde g_B$ is Einstein, and so is $g_B$. 
\end{proof}
\end{prop}

\subsection{Totally geodesic vertical metrics}\label{sec:totgeo}

A class of solutions to the twisted harmonic-Einstein equations are given by triples \((g_B,\rho,h)\) where \(h\) is not just harmonic, but totally geodesic: \(\nabla dh\equiv 0\). These appear to be a natural starting point, since symmetric spaces contain a large number of totally geodesic manifolds given algebraically by Lie triple systems \cite{Helgason78}*{\S~IV.7}. However, they give rise to known examples.

\begin{prop}\label{prop_totgeo}
Let \((M, g, X)\) be an expanding Ricci soliton obtained from a solution $(g_B, \rho, h)$ to the twisted harmonic-Einstein equations \eqref{Eq:HarmonicEinsteinEquations} with totally geodesic $h$. 
Then, the universal cover \((\tilde{M},\tilde{g})\) is the Riemannian product of a homogeneous Ricci soliton and an Einstein manifold.
\begin{proof}

To begin with, assume that \(h\) is an immersion. Observe that \(h\) is necessarily an injective immersion: if \(h_{b_1}=h_{b_2}\) then \(h\) sends any geodesic in \(\tilde{B}\) that joints \(b_1\) to \(b_2\) to a geodesic in \(\sans{G}/\sans{K}\) that starts and ends at the same point. Since \(\sans{G}/\sans{K}\) is simply connected with non-positive curvature, the only way this can occur is if \(b_1=b_2\). Hence, we may identify \(\tilde{B}\) with its image, a totally geodesic submanifold of \( \sans{G}/\sans{K}\).

We claim that \(\tilde{B}\subset \sans{G}/\sans{K}\) is complete in the induced metric $h^* g_{\rm sym}$. Indeed, let \(\tilde{B}=\tilde{B}_0\times\cdots\times \tilde{B}_r\) be the de Rham decomposition \cite{KobayashiNomizu63}*{Thm.~IV.6.2} of \(\tilde{B}\) (for the metric \(g_B\)), where \(\tilde{B}_0\) is a Euclidean space and the holonomy representation of \((\tilde{B}_i,g_i)\) is irreducible for \(i>0\). (Here \(g_i=\tilde{g}_B|_{\tilde{B}_i}\).) Notice that the decomposition \(T\tilde{B}=T\tilde{B}_0\oplus\cdots\oplus T\tilde{B}_r\) is orthogonal for \(h^*g_{\sans{sym}}\) by \eqref{Eq:HarmonicEinsteinEquations(B)}. Moreover, on each \(\tilde{B}_i\), \(i>0\), we have \(h^*g_{\mathsf{sym}}=c_i g_{B}\) for some \(c_i>0\) since \(h^*g_{\mathsf{sym}}\) is parallel and the holonomy acts irreducibly on \(T\tilde{B}_i\). The restriction of \(h^*{g}_{\sans{sym}}\) to \(\tilde{B}_0\) is a Euclidean metric, so up to isometry we may assume that \(h^*g_{\sans{sym}}\) agrees with \(\tilde{g}_0\) on \(\tilde B_0\). Hence, \(h^*g_{\sans{sym}}\) is complete, being the Riemannian product of complete metrics.

Therefore, by \cite{Helgason78}*{Thm.~IV.7.2},  \(\tilde B \subset \G/\K\) is the orbit of a connected reductive subgroup \(\sans{G}'\le \sans{G}\).
Since \(\sans{G}'\) is connected, it preserves the irreducible components of \(\tilde{B}\). We conclude that \(\sans{G}'\) also acts by isometries of \(g_B\). Since \(\sans{G}\) consists of automorphisms of \(\sans{N}\), we can form the semi-direct product \(\sans{G}'\ltimes \sans{N}\). This acts naturally on \(\tilde{M}\simeq\tilde{B}\times \sans{N}\) by \[(\varphi,n)\cdot(b,x) = (\varphi\cdot b, \varphi(n\cdot x)).\] Clearly this action is transitive. It is also isometric since \[\varphi^*(g_B+h_{\varphi(b)}) = g_B+ (\varphi^{-1}\cdot h_{\varphi(b)} )=g_B+h_b.\] Hence, the claim holds in this case.

Suppose now that \(h\) is any \(\rho\)-equivariant totally geodesic map. By \cite{Vilms70}*{Thm.\ 2.2}, we can factor \(h\)  into
\[
        h=h'\circ\eta,
\]
where \(\eta:\tilde{B}\to \tilde{B}'\) is a totally geodesic Riemannian submersion and \(h':\tilde{B}'\to  \sans{G}/\sans{K}\) is a totally geodesic immersion. Since \(\tilde{B}\) is simply connected, we have \cite{EellsLemaire78}*{Par.\ 4.8} that \((\tilde{B},g_B)=(\tilde{B}',g_B')\times (F,g_F)\) is a Riemannian product and \(\eta:(\tilde{B},g_B)\to (\tilde{B}',g_B')\) is the projection onto the first factor. Since \(TF = \ker dh\), \eqref{Eq:HarmonicEinsteinEquations(B)} implies \((F,g_F)\) is Einstein with \(\Ric(g_F) = -\tfrac{1}{2} g_F\). Moreover, it is clear that \((g_B',\rho',h')\) is a solution to the twisted harmonic-Einstein equation with \(h'\) totally geodesic and injective. Letting \((M',g')\) be the Ricci soliton associated to \((g_B',\rho',h')\) via \Cref{prop_AnsatthenHE}, we see that 
\[(\tilde{M},\tilde{g}) \simeq_{\text{isom.}} (\tilde{M}',\tilde{g}')\times (F,g_F).\]
Since \((\tilde{M}',\tilde{g}')\) is homogeneous, the result follows.
\end{proof}
\end{prop}

\begin{example}\label{ex:uniformisingrep}
Let \((B^2,g_B)\) be a closed hyperbolic surface,
% Riemann surface with \(\Ric(g_B)=-g_B\). 
and let \(\rho:\pi_1(B)\to \SL_2(\R)\) be a lift of its holonomy representation to \(\SL_2(\R)\). The developing map \(h:\tilde{B}\to \SL_2(\R)/\SO(2)\) is an isometry. It is also \(\rho\)--equivariant by construction. Hence, by \Cref{prop_AnsatthenHE} and \Cref{prop_totgeo} we obtain a locally homogeneous expanding soliton on \(M^4=\tilde{B}\times_\rho \R^2\) (cf.~\cites{semialglow,Lott10}).
\end{example}

\begin{example}
Let \((B^2,g_B)\) be a genus 2 surface with \(\Ric(g_B) = -\tfrac{1}{2}g_B\). For \(\N=\R^2\), we have \(\G=\det^{-1}(\pm 1)\).
Let \(\rho:\pi_1(B)\to \G\) be the representation defined on a generating set \(\{a_1,b_1,a_{2},b_2\}\) of \(\pi_1(B)\) by \[\rho(a_1) = \begin{pmatrix} 0& 1\\ 1&0 \end{pmatrix},\qquad \rho(a_1)=\rho(b_1)=\rho(b_2)=\begin{pmatrix}  1&0\\ 0&1 \end{pmatrix}.\]
Since \(\rho(\pi_1(B))\le \Or(2)\), the constant map \(h:\tilde{B}\to \det^{-1}(\pm 1)/\Or(2)\), \(h_b= \bar{h}\) is \(\rho\)-equivariant and totally geodesic. Hence, the resulting soliton metric \(g\) on \(M^4=\tilde B\times_\rho \R^2\) is locally isometric to the Riemannian product of \((B^2,g_B)\) with \((\R^2,h_{\rm nil})\), where \(h_{\rm nil}\) is the Euclidean metric. Notice, however, that \((M,g)\) cannot be globally isometric to a product, since the horizontal lift of a closed geodesic in \((B^2,g_B)\) representing \(a_1\) is not closed.
\end{example}

\subsection{One dimensional base}\label{sec:dimB=1}

Since harmonic maps from one-dimensional manifolds are geodesics, \Cref{prop_totgeo}  implies that solutions of the harmonic-Einstein equations in this case  give rise to locally homogeneous Ricci solitons. Homogeneous Ricci solitons are solvsolitons by \cites{BL18,alek_sol}, so the solitons coming from solutions of the twisted harmonic-Einstein equations over a one-dimensional base must be quotients of solvsolitons \cite{solvsolitons}; see also Example~4.27 in \cite{Lot07}. In this situation, we can describe the resulting homogeneous manifold in more familiar terms as follows.

Let \((M,g)\) be a Ricci soliton with a twisted \(\sans{N}\)--action coming from a solution \((g_B,\rho,h)\) of the twisted harmonic-Einstein equations, and such that \(\dim B=1\).  The bundle metric is a harmonic map  \(h:\R \to \G/\K\), thus a geodesic. We may assume that \(h_0 = h_{\rm nil}\) is a nilsoliton.  It is well known \cite{Helgason78}*{Thm.\ IV.3.3} that geodesics in a symmetric space are one-parameter subgroups: 
\[h_s=e^{tD}\cdot h_{\rm nil}= h_{\rm nil}(e^{-tD}\cdot,e^{-tD}\cdot ),
\]
for some \(D\in \mathfrak{p}\), \(\mathfrak{g}=\mathfrak{k}\oplus \mathfrak{p}\) being a Cartan decomposition of \(\mathfrak{g} = {\rm Lie}(\G)\). In our situation, \(\mathfrak{p}\) consists of \(h_{\rm nil}\)-self-adjoint derivations of \(\mathfrak{n}\). It follows that $(M,g)$ is a quotient of the solvable Lie group $\R\ltimes_D \N$ endowed with the \emph{solvsoliton} metric $ds^2 + h_s$ \cite{solvsolitons} (cf.~\Cref{ex:RxN}).

\section{The equations over a Riemann surface}\label{sec_riemsurf}

In this section, we consider the twisted harmonic-Einstein equations \eqref{eqn_twistharmEin} for a pair $(B^2, \G/\K)$, where $B^2$ is a compact orientable surface, and $\G/\K$ is an arbitrary symmetric space of non-positive curvature. If \(c=\{e^{2u}\check{g} : u \in \mathcal{C}^\infty(B)\}\) is a conformal class of metrics on \(B^2\), then the pair \( \Sigma = (B^2,c)\) is a  \emph{Riemann surface}.

Let \((B^2,\check{g})\),  \((N,g_N)\) be Riemannian manifolds and \(h:(B^2,\check{g})\to (N,g_N)\) a smooth map.  Since the source is a surface, if \(h\) is harmonic with respect to \(\check{g}\), then \(h\) is also harmonic with respect to any metric conformal to \(\check{g}\). Therefore, it makes sense to talk about harmonic maps from a Riemann surface $\Sigma$. A map \(h:(B^2,c)\to (N,g_N)\) is called \emph{weakly conformal} if \[h^*g_N = \nu \check{g},\] for some (and hence any) metric \(\check{g}\in c\) and some smooth non-negative function \(\nu:B^2\to \R\). We emphasise that \(\nu\) can have zeros; for example, any constant map is weakly conformal, with \(\nu \equiv 0\). 

\begin{example}
Recall that (after fixing an orientation), a Riemann surface admits a holomorphic atlas. If \(\dim N=2\), then a map \(h:\Sigma \to (N^2,g_N)\) is weakly conformal if and only if \(h\) is holomorphic or anti-holomorphic.
\end{example}

Suppose now that \(\G\) is a real reductive Lie group with maximal compact subgroup \(\K\le \G\). If \((g_B,\rho,h)\) solves the twisted harmonic-Einstein equation \eqref{eqn_twistharmEin} for the pair \((B^2,\G/\K)\), then \(h\) is necessarily weakly conformal by \eqref{Eq:HarmonicEinsteinEquations(B)}. Indeed,  \(\Ric(g_B)=K(g_B) g_B\), where \(K(g_B)\) denotes the Gauss curvature. This suggests the following decoupling of the twisted harmonic-Einstein equation: 
\begin{enumerate}
    \item solve for \(c\) and \(h\) such that \(h:(\tilde B,c)\to (\sans{G}/\sans{K},g_{\sans{sym}})\) is weakly conformal, harmonic and $\rho$-equivariant;
    \item solve the base equation \eqref{Eq:HarmonicEinsteinEquations(B)} within the conformal class \(c\).
\end{enumerate} 

In fact, given a conformal structure on \(B^2\), the existence of a $\rho$-equivariant harmonic map can be characterised purely in terms of \(\rho\in \Hom(\pi_1(B),\G)\):
% it is possible to say for precisely which \(\rho\in \Hom(\pi_1(B),\G)\) there exists a \(\rho\)-equivariant harmonic map.

\begin{theorem}[Corlette, Donaldson, Labourie, \cites{Cor88,Don87,Lab91}]\label{thm_CorDonLab}
Given \(\rho\in \Hom(\pi_1(B),\sans{G})\), there exists a \(\rho\)-equivariant harmonic map \(h_c:(\tilde{B},c)\to (\sans{G}/\sans{K},g_{\sans{sym}})\) if and only if \(\rho\) is reductive. Moreover, \(h_c\) is unique up to left multiplication by elements \(g\in \sans{G}\) that centralise $\rho(\pi_1(B))$.
% commute with the image of \(\rho\).
\end{theorem}

A representation \(\rho\in \Hom(\pi_1(B),\sans{G})\) is called \emph{reductive} if \(\Ad\circ \rho \in \Hom(\pi_1(B), \mathsf{Gl}(\mathfrak{g}))\) 
% is reductive; that is, if \(\mathfrak{g}\) 
decomposes as a sum of irreducible representations. If \(\G\) is algebraic, this is equivalent to the Zariski closure of \(\rho(\pi_1(B))\le \G\) being reductive. We denote the set of reductive representations by \(\Hom^{\sans{red}}(\pi_1(B),\G)\). By \Cref{thm_CorDonLab}, given \(\rho\in \Hom^{\sans{red}}(\pi_1(B),\G)\), the first step in the above decoupling reduces to finding a conformal structure \(c\) for which the harmonic map \(h_c\) is also weakly conformal. Notice that this holds for some $h_c$ if and only if it holds for any, since the action of $\G$ on $(\G/\K, g_{\rm sym})$ is by isometries. It turns out that this condition is also sufficient:

\begin{theorem}\label{Thm:BaseEqn}
If \(h:(\tilde{B}^2,c)\to (\G/\K,g_{\mathsf{sym}})\) is a harmonic, \(\rho\)-equivariant map, then it is possible to solve \eqref{Eq:HarmonicEinsteinEquations(B)} for \(g_B\in c\) if and only if \(h\) is weakly conformal. Moreover, when it exists, such a solution is unique.
\begin{proof}

We have already observed the necessity of \(h\) being weakly conformal, so let us focus on suffiency. Let \(\check{g}\in c\) be the hyperbolic metric and assume that \(h^*g_{\mathsf{sym}}=\nu \check{g}\) for some function \(\nu\ge 0\). Since \(h\) is \(\rho\)-equivariant and \(g_{\sans{sym}}\) is \(\G\)-invariant, the pullback \(h^*g_{\sans{sym}}\) is \(\pi_1(B)\)-invariant. Thus, \(h^*g_{\sans{sym}}\) is the pullback of a tensor on \(B^2\) via the covering map. We will abuse notation and use the same symbol for the tensor on \(B^2\) that pulls back to \(h^*g_{\sans{sym}}\).

We must show that \eqref{Eq:HarmonicEinsteinEquations(B)} holds for some \(g_B\in c\). Any metric in \(g_B\in c\) can be written as \(g_B=e^{2u}\check{g}\) for some \(u:B\to \R\). The Gauss curvature of \(g_B\) is \[K(g_B) = e^{-2u}(-\Delta_{\bar{g}} u - 1).\] Hence, \eqref{Eq:HarmonicEinsteinEquations(B)} reduces to the scalar equation \begin{equation}\label{Eq:ScalarTwistedEinstein}
 \Delta_{\check{g}} u =\tfrac12 e^{2u} - 1 -\nu .
\end{equation}
To solve \eqref{Eq:ScalarTwistedEinstein}, consider the functional \(I:H^1(B)\to \R\) defined by \[I[u] = \int_B |d u|^2 - u\big(1+\nu\big)\, {\rm dvol}_{\check{g}}.\] We minimise this subject to the Gauss-Bonnet constraint \begin{equation}\label{Eq:Constraint}
 2\pi\chi(B)+\frac{1}{2}\int_B e^{2u}\, {\rm dvol}_{\check{g}} =\int_B \nu \,{\rm dvol}_{\check{g}}
\end{equation} 
Observe that the set of functions \(u:B\to \R\) satisfying \eqref{Eq:Constraint} is non-empty since \(\chi(B)<0\). It then follows, in the exact same way as for the prescribed Gauss curvature equation, that such a minimiser exists and, moreover, is a weak solution of \eqref{Eq:HarmonicEinsteinEquations(B)}; see \cite{Taylor11}*{Ch.~14 $\S$2}. Elliptic regularity then shows that the weak solution is a classical solution and is smooth. Conformally related metrics on surfaces have the same Ricci curvature if and only if (the logarithm of) their conformal factor is harmonic; since \(B^2\) is compact, any harmonic function is constant. This constant must be 1 by \eqref{Eq:HarmonicEinsteinEquations(B)}, so the solution is unique. 
\end{proof}
\end{theorem}

We finish this section by recalling one approach to finding weakly conformal harmonic maps due to Schoen--Yau \cite{SchoenYau79} and Sacks--Uhlenbeck \cite{SacksUhlenbeck82}. It has been extended to the twisted setting by Labourie \cite{Labourie08}. For \(\rho\)-equivariant maps \(h:(\tilde B, c)\to (\G/\K,g_{\sans{sym}})\), the tensor \(h^*g_{\sans{sym}}\) on \(\tilde B\) is \(\pi_1(B)\)-invariant, and hence defines a tensor on \(B^2\). In particular, the energy density \(e_{\check{g}}(h) = \tr_{\check{g}} h^*g_{\sans{sym}}\) is a function on \(B^2\). Since \(B^2\) is compact, \(h:(\tilde B, c)\to (\G/\K,g_{\sans{sym}})\) has a well-defined energy: \[E(h) = \int_B e_{\check{g}}(h)\, {\rm d vol}_{\check{g}}.\] The \(2\)-form \( e_{\check{g}}(h)\, {\rm d vol}_{\check{g}}\) is conformally invariant, so \(E\) depends only on the conformal class \(c\) and not on the choice of \(\check{g}\in c\). Thus, we can define a function \(E:{\rm Teich}(B)\to \R_{\ge 0}\) by
\[
    E_\rho(c) = \inf \{ E(h) : h \,\, \rho-\hbox{equivariant} \},
    % \qquad \rho(\gamma)\cdot h = h\circ \gamma, \, \forall \gamma\in \pi_1(B),
\]
where \({\rm Teich}(B)\) is the \emph{Teichm\"uller} space of \(B^2\), the space of marked conformal structures on \(B^2\). The infimum is attained \cite{Labourie08}*{Thm.\ 5.2.1} for reductive representations \(\rho\) by the \(\rho\)-equivariant harmonic map \(h_c:(\tilde B^2,c)\to (\G/\K,g_{\sans{sym}})\) from Theorem~\ref{thm_CorDonLab}. Moreover, if \(c\) minimises \(E_\rho\), then \(h_c\) is weakly conformal and harmonic \cite{SacksUhlenbeck82}*{Theorem~1.2}.

If \(\Gamma\le \G\) is a torsion-free, co-compact lattice and \(\rho(\pi_1(B))\le \Gamma\), then a \(\rho\)-equivariant map \(h:\tilde B\to \G/\K\) defines a genuine map \(\check{h}:B\to \Gamma\backslash \G/\K\) between compact manifolds. It follows from \cites{SchoenYau79,SacksUhlenbeck82} that if \(\rho\) is injective, then there is a conformal structure \(c\) on \(B^2\) such that \(E_\rho(c) = \inf_{c'}E_\rho(c')\). Hence, for injective representations \(\rho\) that map into torsion-free, co-compact lattices, we can always solve the twisted harmonic-Einstein equation \eqref{Eq:HarmonicEinsteinEquations(B)} for some conformal structure \(c\). Note that the existence of surface subgroups of cocompact lattices in semisimple Lie groups is known in several cases, for example, in complex semisimple groups and simple rank one symmetric spaces other than \(\SO(2n,1)\); see the survey \cite{Surface_grps_in_lattices} and references therein.

In the twisted setting, Labourie \cite{Labourie08}*{Thm.\ 1.0.3} proved that \(E_\rho\) is a proper function on Teichm\"uller space when \(\rho\) is a \emph{Hitchin representation}. A Hitchin representation is a homomorphism \(\rho \in \Hom(\pi_1(B),\G)\), where \(\G\) is the adjoint form of a split semi-simple Lie group, that can be continuously deformed to \(\iota\circ \rho_{\rm Fuch}\), where \(\rho_{\rm Fuch}:\pi_1(B)\to \PSL_2(\R)\) is the holonomy of a hyperboilic structure on \(B^2\) and \(\iota:\PSL_2(\R)\to \G\) is an irreducible representation. Note that Hitchin representations are automatically irreducible \cite{Labourie08}*{Prop.\ 4.1.5}, and the set of all Hitchin representations forms a union of connected components in the character variety \(\Hom^{\sans{red}}(\pi_1(B),\G)/\G\), so-called \emph{Hitchin components}.

\begin{proof}[Proof of \Cref{maincor_Hit_comp}]
Let \(\rho\in \Hom(\pi_1(B),\SL_n(\R))\) be a representation such that, for \(\check{\rho }=\Ad\circ \rho\),  the class \([\check{\rho}]\in\Hom^{\sans{red}}(\pi_1(B),\PSL_n(\R))/\PSL_n(\R)\) is in the Hitchin component. By \cite{Labourie08}*{Corollary~1.0.4}, there is a conformal structure \(c\) on \(B^2\) and a \(\check{\rho}\)-equivariant map \(h:(\tilde B,c)\to (\PSL_n(\R)/\mathsf{PSO}(n),g_{\sans{sym}})\) that is harmonic and weakly conformal. But \(h\) is also \(\rho\)-equivariant since the action of \(\rho\) on \(\SL_n(\R)/\SO(n) \simeq \PSL_n(\R)/\mathsf{PSO}(n)\) factors through \(\SL_n(\R)\). 

By Theorem~\ref{Thm:BaseEqn}, there is a unique metric \(g_B\in c\) such that \((g_B,\rho,h)\) solves the twisted harmonic-Einstein equations. By Proposition~\ref{prop_AnsatthenHE}, there is a soliton metric on \(M=\tilde B\times _\rho \R^n\).
\end{proof}

\subsection{Higgs Bundles}

By \Cref{Thm:BaseEqn}, solving the twisted harmonic-Einstein equations over a surface reduces to identifying pairs \((c,\rho)\), where 
\(c\) is a conformal structure on \(B^2\) and 
\(\rho\in \Hom^{\sans{red}}(\pi_1(B),\G)\), such that the \(\rho\)-equivariant harmonic map \(h:(\tilde B,c)\to (\G/\K,g_\sans{sym})\) is also weakly conformal. In order to do that, we recall the non-abelian Hodge correspondence, which identifies \(\Hom^{\sans{red}}(\pi_1(B),\G)/\G\) with the moduli space of polystable \emph{$\G$-Higgs bundles}. 

In this subsection, \(\G\) will denote a connected linear real reductive group. We fix a faithful representation \(\G\le \GL_n(\R)\) such that \(\G\) is closed under transposition. This gives the Cartan decomposition \[\G =\K \exp(\mathfrak{p}),\] where \(\K=\G \cap \SO(n)\) and \(\mathfrak{p} = \mathfrak{g}\cap \sym(n)\). We follow \cite{SOpq_Higgs_paper}*{\S2} for the background we need regarding \(\G\)-Higgs bundles. 
% The advantage of working on the Higgs bundle side of the correspondence for us is that there is a natural, gauge-invariant condition that detects when the \(\rho\)-equivariant map \(h:(\tilde B,c)\to (\G/\K,g_\sans{sym})\) is weakly conformal. 

\begin{definition}
    A \(\G\)-Higgs bundle is a pair, \((P,\Phi)\), where \(P\to \Sigma\) is a holomorphic \(\K^\C\)-principal bundle and \(\Phi\) is a holomorphic section of the associated bundle \(K_\Sigma\otimes (\ad_P( \mathfrak{p}^\C))\).
\end{definition}

The Hodge star on one-forms is a conformal invariant and hence makes sense on a Riemann surface \(\Sigma = (B^2,c)\). Since it squares to \(-\id_{{T^*\!}B}\), the complexified cotangent bundle \(T^*B\otimes \C\) splits into \(\pm i\)-eigenbundles. The \emph{canonical bundle}, which we denote by \(K_\Sigma\), is the \(+i\)-eigenspace. The vector bundle $\ad_P(\mathfrak{p}^\C) = P\times_{\K^\C} \mathfrak{p}^\C$ is associated to $P$ via the adjoint representation of $\K^\C$ on $\mathfrak{p}^\C$. In a local complex coordinate $z$, the holomorphic sections of $K_\Sigma$ are given by \(f dz\), where \(f\) is a holomorphic function. Thus, in local coordinates, a holomorphic section of \(K_\Sigma\otimes (\ad_P( \mathfrak{p}^\C))\) is given by \(A \, dz\), where \(A\) is a holomorphic map into \(\mathfrak{p}^\C\).

\begin{example}[\(\SL_n(\C)\)-Higgs bundles, \cite{SOpq_Higgs_paper}*{Def.\ 2.4}]
    We claim that an \(\SL_n(\C)\)-Higgs bundle \((P,\Phi)\) can be described equivalently in terms of a \emph{Higgs vector bundle}: a pair \((E,\Phi)\) where \(E\) is a holomorphic vector bundle with \(\Lambda^nE\simeq \mathcal{O}_\Sigma \) and \(\Phi:E\to K_\Sigma\otimes E\) is a trace-free endomorphism twisted by \(K_\Sigma\). 
    
    Indeed, given a holomorphic \(\K^\C= \SL_n(\C)\)-bundle \(P\), we can consider the associated vector bundle \(E=P\times_{\SL_n(\C)}\C^n\) corresponding to 
    % associated to \(P\), where \(\SL_n(\C)\) acts by 
    the defining representation. Notice that \(\Lambda^nE\simeq \mathcal{O}_\Sigma\) is holomorphically trivial. Conversely, given a holomorphic vector bundle with \(\Lambda^nE\) holomorphically trivial, its bundle of frames preserving a holomorphic volume form is a holomorphic principal \(\SL_n(\C)\)-bundle \(P\). Since the adjoint representation of \(\SL_n(\C)\) is the restriction of the tensor product representation to the subspace \(\mathfrak{sl}_n(\C)\subset \mathfrak{gl}_n(\C)\simeq \C^n\otimes (\C^n)^*\), there is a bundle isomorphism \(P\times_{\Ad}\mathfrak{sl}_n(\C) \simeq \End_0(E)\), where \(\End_0(E)\) is the bundle of trace-free endomorphisms of \(E\). Under this isomoprhism, the Higgs field corresponds to a $K_\Sigma$-twisted endomorphism.  
\end{example}

If \(\G\le \SL_n(\R)\) is linear, then we can still associate a Higgs vector bundle to any \(\G\)-Higgs bundle. We find that Higgs vector bundle are conceptually easier to work with since they are linear objects. Because of this, we will also adopt the linear algebra notation for \(\Phi\) (e.g.\ \(\tr \Phi =0\)).

%\(\G\)-Higgs bundles also admit the following description as a vector bundle with a ``twisted'' endomorphism field. Let \(E = P\times_{\K^\C}\C^n\) be the holomorphic vector bundle associated to \(P\), where \(\K^\C\le \GL_n(\C)\) acts on \(\C^n\) via the standard representation. The Higgs field then becomes a holomorphic section of \(\End(E)\otimes K_\Sigma\) since there is an inclusion \(P\times_{\Ad}\mathfrak{p}^\C \subset P\times_{\Ad} \mathfrak{gl}_n(\C) = \End(E)\). For this reason, we will sometimes use ``linear algebra'' notation. For example

\begin{example}[\(\SL_2(\R)\)-Higgs bundles, \cite{Hitchin87}*{\S10}]\label{ex:SL_2(R)-Higgs}
For \(\G=\SL_2(\R)\), we have \(\K^\C = \SO(2,\C)\simeq\C^*\). Let  $P \to \Sigma$ be a holomorphic principal $\SO(2,\C)$-bundle. The group $\C^*$ embeds into \(\SL_2(\C)\) by \(a\mapsto \operatorname{diag}(a,a^{-1})\). For this embedding, we have \[\mathfrak{p}^\C = \bigg\{\begin{pmatrix}
    0& b \\ c& 0
\end{pmatrix} :\ b,c\in \C\bigg\}\simeq \C^{\otimes 2}\oplus (\C^*)^{\otimes 2}\ \ \text{as }\ \C^*\text{-modules}.\]
Here, \(\C^*\) acts on \(\C^{\otimes k}\) by the tensor product representation induced from the standard representation of \(\C^*\) on \(\C\) (i.e.\ if we identify \(\C^{\otimes k}\) with \(\C\) as a vector space, then \(a\cdot v = a^k v\) for \(a\in \C^*,\, v\in \C^{\otimes k}\)). 
Then, \[K_\Sigma\otimes(P\times_{\K^\C} \mathfrak{p}^\C)\simeq K_\Sigma L^2 \oplus K_\Sigma L^{-2},\] where \(L\) is the line bundle associated to \(P\) via the standard representation of \(\C^*\) on \(\C\). Thus, an \(\SL_2(\R)\)-Higgs bundle is a triple \((L,\beta,\gamma)\) where \(L\) is a holomorphic line bundle, \(\beta \in H^0(B,K_\Sigma L^2)\), and \(\gamma \in H^0(B,K_\Sigma L^{-2})\). 
\end{example}

%Fix a \(\G\)-Higgs bundle, \((P,\Phi)\) and suppose we are given a metric \(h\) on \(P\) (i.e.\ a holomorphic reduction of the structure group to \(\K\le \K^\C\)). Let \(P_{\K}\le P\) be the reduced bundle. The bundle \(P_\K\) has a canonical connection: the Chern connection 

Recall that the degree of a complex vector bundle is given by \[\deg E = \frac{-1}{2\pi i}\int_B \tr \Omega \, \in \, \Z,\] where \(\Omega\) is the curvature of any connection on \(E\). An \(\SL_n(\C)\)-Higgs bundle is called \emph{semistable} if the associated vector bundle \(E\) satisfies \[\deg E' \le 0,\] for all \(\Phi\)-invariant subbundles \(E'\subset E\). (Here, \(\Phi\)-invariant means \(\Phi(E')\subset K_\Sigma\otimes E'\).) An \(\SL_n(\C)\)-Higgs bundle is \emph{stable} if the inequality is strict for all \(E'\), and it is \emph{polystable} if it is semistable and it splits as a direct sum of stable bundles (all of degree \(0\)).

There is a more general notion of polystability for \(\G\)-Higgs bundles that makes the non-abelian Hodge correspondence true for general connected real reductive Lie groups; see \cite{GGM}*{Definition~2.9}. However, it turns out that for linear groups, this is equivalent \cite{SOpq_Higgs_paper}*{Proposition~2.10} to polystability for Higgs vector bundles defined above:

\begin{prop}\label{prop_polystable}
Let $\G$ be a linear real reductive Lie group. A \(\G\)-Higgs bundle \((P,\Phi)\) is polystable if and only if the associated Higgs vector bundle is polystable.
\end{prop}

We then have the following non-abelian Hodge correspondence. This was proved by Hitchin \cite{Hitchin87} and Donaldson \cite{Don87} for \(\SL_2(\C)\), by Simpson \cite{Simp88} and Corlette \cite{Cor88} for complex reductive groups, and by Garc\'ia-Prada--Gothen--Mundet I Riera \cite{GGM} for connected real reductive groups. 

\begin{theorem}[Non-abelian Hodge correspondence]\label{thm:nonabHodge}
   For each conformal structure \(c\) on \(B^2\), there is a homeomorphism \[\mathcal{M}(\Sigma,\G)\simeq \Hom^{\sans{red}}(\pi_1(B),\G)/\G,\] where \(\mathcal{M}(\Sigma,\G)\) is the space of \(\K^\C\)-gauge equivalence classes of polystable \(\G\)-Higgs bundles, $\Sigma = (B,c)$.
\end{theorem}

We note that this correspondence is $\Aut(\Sigma)$-equivariant. Here, $\Aut(\Sigma)$ acts on $\mathcal{M}(\Sigma,\G)$ by pull-back, and on $\Hom^{\sans{red}}(\pi_1(B),\G)/\G$ by pre-composition with the induced action on the fundamental group.

In order to explain the gauge-invariant condition on Higgs bundles that detects conformality of \(\rho\)-equivariant harmonic maps, we describe briefly how a reductive representation \(\rho\in\Hom^{\sans{red}}(\pi_1(B),\G)\) gives rise to a Higgs bundle.

Given \(\rho\in \Hom^{\sans{red}}(\pi_1(B),\G)\), let \(h: \tilde \Sigma \to (\G/\K , g_{\sans{sym}})\) be a \(\rho\)-equivariant harmonic map (which exists by Theorem~\ref{thm_CorDonLab}).
Let \(E_\R=\tilde \Sigma\times_\rho \R^n\) be the flat vector bundle associated to \(\rho\) (recall that \(\rho(\pi_1(B))\le \G\le \SL_n(\R)\)), and \(E=E_\R\otimes \C\) its complexification. Then \(h:\tilde \Sigma\to \G/\K\subset \SL_n(\R)/\SO(n)\) defines a bundle metric on \(E_\R\), and its differential \(dh\) defines a section \(h^{-1}dh\in\Omega^1(\Sigma,\End_0(E_\R))\). The Higgs field \(\Phi\) is the \((1,0)\)-part of the complex-linear extension \(h^{-1}dh^\C \in \Omega^1_\C(\Sigma,\End_0(E))\). Since \(h\) is harmonic, the bundle \(E\) admits a holomorphic structure such that \(\Phi\) is a holomoprhic section of \(K_\Sigma\otimes \End(E)\) (cf.\ \cite{SOpq_Higgs_paper}*{Appendix~A}). From this, we obtain the following lemma which is well-known to experts.

\begin{lemma}\label{lem_hopfdif}
Let \(\rho\in \Hom^{\sans{red}}(\pi_1(B),\G)\) be a reductive representation, \(h:\tilde \Sigma \to (\G/\K,g_{\sans{sym}})\) a \(\rho\)-equivariant harmonic map. If \((P,\Phi)\) is a \(\G\)-Higgs bundle representing the gauge equivalence class corresponding to \(\rho\) under the non-abelian Hodge correspondence, then \(\tr \Phi^2=0\in H^0(\Sigma,K_\Sigma^2)\) if and only if \(h\) is weakly conformal.

\begin{proof}
Recall that the complexified cotangent bundle can be decomposed as \(T^* \Sigma^\C = T^* \Sigma \otimes \C=K_\Sigma\oplus \bar{K}_\Sigma\), where \(K_\Sigma\) is the canonical bundle. There is a corresponding decomposition of \(\Sym^2(T^*\Sigma)^\C\) as \(K_\Sigma^2\oplus K_\Sigma\bar{K}_\Sigma\oplus \bar{K}_\Sigma^2\).
%into a \((2,0)\)-part, a \((1,1)\)-part, and a \((0,2)\)-part. 
Given a real symmetric tensor \(\alpha\in \Sym^2(T^*B)\), let \(\alpha^\C = \alpha^{2,0}+\alpha^{1,1}+\alpha^{0,2}\) be the decomposition of its complex-linear extension into these components. Then, \(\alpha\) is conformal to \(\check{g}\in c\) if and only if \(\alpha=\alpha^{1,1}\). Since \(\alpha\) is real, \(\overline{\alpha^{2,0}} = \alpha^{0,2}\), hence  \(\alpha=\alpha^{1,1}\) if and only if \(\alpha^{2,0}=0\).

Thus, the lemma will follow if we show that \(\tr \Phi^2= (h^*g_{\sans{sym}}^\C)^{2,0}\). But the complexification of \(h^{-1}dh\) is \(h^{-1}dh^\C =\Phi +\Phi^*\), so the decomposition of \(h^*g_{\sans{sym}}^\C\) is \[h^*g_{\sans{sym}}^\C = \tr \Phi^2 +\tr \Phi\Phi^* +\tr (\Phi^*)^2,\] since \(g_{\sans{sym}}(A,A) = \tr A^2\) is the restriction of the Killing form. This proves the claim.
\end{proof}
\end{lemma}

\begin{proof}[Proof of \Cref{main:tHE-Higgs}]
Fix a conformal class \(c\) in \(B^2\) and set $\Sigma = (B,c)$. Notice that the map \((P,\Phi)\mapsto \tr \Phi^2\) is gauge invariant, so it makes sense to define the closed subset \[\mathcal{N}=\{[(P,\Phi)]:\tr \Phi^2 =0\}\subset \mathcal{M}(\Sigma,\G).\]
Notice also that  $\Aut(B^2,c)$ acts on $\mathcal{N}$ by pull-back. We also let $\mathcal{HE}$ denote the set of solutions $(g_B, \rho, h)$  to the twisted harmonic-Einstein equations with \(g_B\in c\), up to equivalence (\Cref{def:tHE_equivalence}). 

We first define a map $\mathcal{N} \to \mathcal{HE}$ which is invariant under the action of $\Aut(B^2,c)$ on $\mathcal{N}$. 
By \Cref{thm:nonabHodge}, a given $[(P,\Phi)] \in \mathcal{N}$ corresponds to some \(\rho\in \Hom^{\sans{red}}(\pi_1(B),\G)\) which is well-defined only up to conjugation by  $\varphi \in \G$. By \Cref{thm_CorDonLab}, there exists a $\rho$-equivariant harmonic map \(h:(\tilde{B},c)\to (\G/\K,g_{\sans{sym}})\), which is weakly conformal thanks to \Cref{lem_hopfdif}. \Cref{Thm:BaseEqn} yields the existence of a unique metric \(g_B\in c\) such that \eqref{Eq:HarmonicEinsteinEquations(B)} holds, and hence \((g_B,\rho,h) \in \mathcal{HE}\). If we had a different representative $C_\varphi \circ \rho \in \Hom^{\sans{red}}(\pi_1(B),\G)$, $\varphi \in \G$, then the corresponding triple would clearly be $(g_B, C_\varphi \circ \rho, \varphi \cdot h) \sim (g_B, \rho, h)$. On the other hand, given  $\check f \in \Aut(B^2,c)$, then $\check f^* [(P, \Phi)]$ corresponds to $\rho \circ \check f_* \in \Hom^{\sans{red}}(\pi_1(B),\G)$. The triple associated to $\rho \circ \check f_* $ is clearly $(\check f^* g_B, \rho \circ \check f_*,  h \circ \check f)$, which is also equivalent to $(g_B, \rho, h)$.

Conversely, if \((g_B,\rho,h) \in \mathcal{HE}\),
% solves the twisted harmonic-Einstein equation with \(g_B\in c\), 
then Theorem~\ref{thm_CorDonLab} implies \(\rho\in \Hom(\pi_1(B),\G)\) is reductive, so the non-abelian Hodge correspondence gives \([(P,\Phi)]\in \mathcal{M}(B,\G)\). Since \(h\) is weakly conformal by \eqref{Eq:HarmonicEinsteinEquations(B)}, Lemma~\ref{lem_hopfdif} yields \(\tr \Phi^2 =0\), so \([(P,\Phi)]\in \mathcal{N}\). Finally, if \((g_B',\rho',h') = (\check f^* g_B,C_\varphi \circ \rho \circ \check f_*,\varphi\cdot h \circ \check f ) \sim (g_B, \rho ,h)\), then \(\rho' =C_\varphi \circ \rho \circ \check f_*\) defines the same conjugacy class as \(\rho\) (up to pull-back), so \((g'_B,\rho',h')\) defines the same equivalence class \( \check f^* [(P,\Phi)]\in \mathcal{N}\) as \((g_B,\rho ,h)\).
\end{proof}

\subsection{Solitons over Riemann surfaces: Proof of \Cref{maincor:isom_of_sols}}\label{sec:solitons_higgs}

In this subsection we briefly return to considering Ricci solitons constructed using Ansatz~\ref{Ansatz} via Proposition~\ref{prop_AnsatthenHE}. Our main aim is to prove Corollary~\ref{maincor:isom_of_sols}. We also prove a necessary condition for solitons satisfying Ansatz~\ref{Ansatz} to be homogeneous. 

Let \(\Sigma=(B^2,c)\) be a Riemann surface, with \(\chi(B^2)<0\). Combining Theorem~\ref{main:tHE-Higgs} with Proposition~\ref{prop_AnsatthenHE}, we can associate a Ricci soliton to a polystable Higgs bundle. Notice that we must make several choices for this association, but different choices lead to isometric solitons. Moreover, changing the Higgs bundle by a gauge transformation also does not change the isometry type of the soliton.

 Theorem~\ref{main:tHE-Higgs} has the following drawback for producing examples of solitons from Ansatz~\ref{Ansatz}: if \(\G\) is the reductive group associated to a nilsoliton \((\N,\bar h)\), then we must apply Theorem~\ref{main:tHE-Higgs} to Higgs bundles with \(\G^0\)-Higgs bundles since the non-abelian Hodge correspondence (cf.\ Theorem~\ref{thm:nonabHodge}) is currently only known for connected groups. Hence, we only produce solitons \((M,g)\) such that \(\rho(\pi_1(B))\le \G^0\), where \(\rho\) is the twisting homomorphism of \(\pi:M\to B\). However, this is not a major drawback since we do produce everything up to taking a finite cover. Indeed, the cover \(\hat B\to B\) such that \(\pi_1(\hat{B})=\ker \bar \rho\), where \(\bar \rho\) is the composition of \(\rho\) with the projection \(\G\to \G/\G^0\), is finite and the twisting homomorphism of the pullback of \(M\) to \(\hat B\) maps into \(\G^0\). (The cover is finite since the deck transformations inject into the finite group \(\G/\G^0 = \K/\K^0\).) 
 
By Theorem~\ref{main:tHE-Higgs}, gauge equivalence classes of \(\G^0\)-Higgs bundles are in one-to-one correspondence with solutions of the twisted harmonic-Einstein equation for the pair \((B^2,\G^0/\K^0)\) up to equivalence. Moreover, solutions of the twisted harmonic-Einstein equation for the pair \((B^2,\G/\K)\) are in one-to-one correspondence with solitons by \Cref{prop:equiv_iff_isom}. Hence, to prove Corollary~\ref{maincor:isom_of_sols} it suffices to show the following:

\begin{prop}
    Let \(\Sigma=(B^2,c)\) be a Riemann surface, \(\G\) a reductive group, and \(\G^0\le \G\) its identity component. Moreover, let \(\mathcal{HE}\) (resp.\ \(\mathcal{HE}_0\)) denote the set of solutions $(g_B, \rho, h)$ to the twisted harmonic-Einstein equations, up to equivalence, for the pair \((B^2,\G/\K)\) (resp.\ \((B^2,\G^0/\K^0)\)) with \(g_B\in c\) and \(\rho(\pi_1(B))\le \G^0\). Then, the map \(\mathcal{HE}_0\to \mathcal{HE}\) is finite-to-one.
\begin{proof}
    Let $(g_B, \rho, h)$ be a solution to the twisted harmonic-Einstein equation with \(g_B\in c\) and \(\rho(\pi_1(B))\le \G^0\). Up to pulling back by \(\check{f}\in \Aut(\Sigma)\), any solution \((g_B',\rho',h')\) that is equivalent to $(g_B, \rho, h)$ for \((B^2,\G/\K)\) is of the form $(g_B, C_\varphi\circ\rho,\varphi\cdot h)= (g_B,\rho',h')$ for \(\varphi\in \G\). Notice that this equivalence class is a union of equivalence classes for the pair \((B^2,\G^0/\K^0)\). Since the map \(\mathcal{HE}_0\to \mathcal{HE}\) sends the \(\G^0\)-equivalence class of \((g_B, \rho, h)\) in \(\mathcal{HE}_0\) to the \(\G\)-equivalence class of \((g_B, \rho, h)\) in \(\mathcal{HE}\), the proposition will follow if we show that this union is finite.
    
    Let \(\{\varphi_0=e,\varphi_1,\ldots,\varphi_N\}\subset \G\) be representatives for the \(\G^0\)-cosets in \(\G\). Any $(g_B, C_\varphi\circ\rho,\varphi\cdot h)= (g_B,\rho',h')$ is in the \(\G^0\)-equivalence class of $(g_B, C_{\varphi_i}\circ\rho,\varphi_i\cdot h)$ for some \(i\in \{0,\ldots,N\}\). That is, the number of \(\G^0\)-equivalence classes in any \(\G\)-equivalence class is at most \(N\). Thus, the proposition follows.
\end{proof}
\end{prop}

Finally, we prove the following necessary condition for a soltion \((M,g)\) that satisfies Ansatz~\ref{Ansatz} to be homogeneous. We will use the contrapositive of this in \S4 to prove existence of inhomogeneous solitons coming from Ansatz~\ref{Ansatz}.

\begin{lemma}\label{lem_suff_for_homogeneity}
If a soliton \((M,g)\) satisfying Ansatz~\ref{Ansatz} is locally homogeneous, then \((B,g_B)\) is also locally homogeneous.
\begin{proof}
    Since \(\mathcal{H}\) is flat and \(\rho(\pi_1(B))\le \Aut(\sans{N})\), there is a totally geodesic section \(\sigma:(B,g_B)\to (M,g)\). Thus, local homogeneity of \((M,g)\) implies local homogeneity of \((B,g_B)\) since a totally geodesic submanifold of a locally homogeneous space is itself locally homogeneous. 
\end{proof}
\end{lemma}

\section{Low-dimensional solitons}\label{sec:4D}

\subsection{Dimension 4}
In this section, we apply the results of \S\ref{sec_riemsurf} to give a complete description of the Ricci solitons arising as limits of four-dimensional immortal Ricci flows. It turns out that the only non-trivial examples (i.e.\ not Einstein or locally homogeneous) have \(\dim B=\dim \N=2\); see \cite{Lott10}*{Prop.~4.80}.

In order to do this, we use an explicit description of \(\SL_2(\R)\)-Higgs bundles due to Hitchin \cite{Hitchin87}. 
Fix a Riemann surface \(\Sigma=(B^2,c)\) with \(\chi(B^2)<0\) and let \((L,\beta,\gamma)\) be an \(\SL_2(\R)\)-Higgs bundle (cf.\ Example~\ref{ex:SL_2(R)-Higgs}). For each \(d\in \Z\), we define \(\mathcal{N}_d\subset \mathcal{N}\) to be the subset of equivalence classes of polystable \(\SL_2(\R)\)-Higgs bundles with \(\tr\Phi^2=0\) and \(\deg L=d\). If \((L,\beta,\gamma)\) is polystable, then \(|d|<g-1\) (cf.\ the proof of \cite{Hitchin87}*{Prop.\ 10.2}), so we assume this from now on. Moreover, there is an \(\det^{-1}(\pm 1)\)-gauge transformation (see \cite{Collier20}*{\S 3.2}) that gives an isomorphism of \((L,\beta,\gamma)\) with \((L^{-1},\gamma,\beta)\), so we may assume \(d\ge 0\).

%line bundle of degree \(d\). Given \(\beta \in H^0(B,K_BL^2)\) and \(\gamma \in H^0(B,K_BL^{-2})\), we obtain (cf.\ Example~\ref{ex:SL_2(R)-Higgs}) . 
The Higgs vector bundle corresponding to \((L,\beta,\gamma)\) is \(E=L\oplus L^{-1}\) with the Higgs field \[\Phi =\begin{pmatrix}
    0& \beta \\\gamma &0
\end{pmatrix}.\]
The condition \(0=\tr \Phi^2 = \beta\gamma+\gamma\beta\) implies that either \(\beta=0\) or \(\gamma=0\). If \(d>0\), polystability implies \(\gamma\neq 0\) since \(L\) would be \(\Phi\)-invariant if \(\gamma=0\). Hence, \(\beta=0\) in this case. From Hitchin's desciption \cite{Hitchin87}*{Prop.\ 10.2} of the moduli space of \(\SL_2(\R)\)-Higgs bundles, we obtain the following description of \(\mathcal{N}_d\), \(0<d\le g-1\).

\begin{prop}[Hitchin, \cite{Hitchin87}]
For a fixed Riemann surface \(\Sigma=(B^2,c)\) and an integer \(d\in (0,g-1]\), the set \(\mathcal{N}_d\) is a \(2^{2g}\)-fold covering of \(S^{r}B\), where \(r=2g-2-2d\). 
\end{prop}

Here, \(S^rB\) is the \(r\)-th symmetric product of \(B\). Recall that points \(D\in S^r B\) can be thought of as formal sums \(D=n_1b_1+\cdots +n_k b_k\) with \(b_i\in B\) for \(i=1,\ldots, k\) and \(r=n_1+\cdots+n_k\). The Higgs bundle \((L,0,\gamma)\) lies above the zero divisor of \(\gamma\); that is, the formal sum determined by the points \(b_i\in B\) such that \(\gamma_{b_i}=0\), weighted by the order to which \(\gamma\) vanishes. In particular,  the real dimension of \(S^rB\) is \(2r\).

For \(d=g-1\), \(\mathcal{N}_d\) is a finite set of order \(2^{2g}\). The corresponding representations \(\rho:\pi_1(B)\to \SL_2(\R)\) are the \(2^{2g}\)-possible lifts of the uniformising representation \(\pi_1(B)\to \sans{PSL}_2(\R)\), and the corresponding solitons are homogeneous (cf.\ Example~\ref{ex:uniformisingrep}). 

If \(d=0\), the condition \(\tr \Phi^2 =0\) forces \(\Phi =0\). Indeed, if \(\beta =0\), then \(L^{-1}\) is a \(\Phi\)-invariant sub-bundle with \(\deg L^{-1}=0\), so \(E=L\oplus L^{-1}\) must split as a sum of Higgs line bundles of degree \(0\). The only way this is possible is if \(\gamma=0\). Similarly, \(\gamma=0\) implies \(\beta=0\). Thus, if \(\rho:\pi_1(B)\to \SL_2(\R)\) corresponds to \((L,\beta,\gamma)=(L,0,0)\in \mathcal{N}_0\), then \(\rho(\pi_1(B))\le \Or(2)\) (see, for example, \cite{Collier20}*{Prop.\ 2.17}) and \(\rho\)-equivariant harmonic maps are constant. In particular, the resulting solitons are locally homogeneous.

If \(0<d<g-1\), then the corresponding solitons are not locally homogeneous:

\begin{lemma}
    If \(0<d<g-1\), then a soliton \((M,g)\) associated to an \(\SL_2(\R)\)-Higgs bundle \((L,0,\gamma)\) with \(\deg L=d\) is not locally homogeneous. 
    \begin{proof}
         Assume for a contradiction that \((M,g)\) is locally homogeneous. Then, \((B^2,g_B)\) is locally homogeneous by Lemma~\ref{lem_suff_for_homogeneity}, and hence has constant Gauss curvature since \(\dim B=2\). By \eqref{Eq:HarmonicEinsteinEquations(B)}, \(h^*g_B=\Ric_{g_B}+\tfrac12g\) is a constant multiple of \(g_B\). In particular, \(h^*g_{\sans{sym}}\), and hence also \(h^*g_{\sans{sym}}^\C =\tr\Phi\Phi^*\), either vanishes identically or vanishes nowhere. The claim now follows since \(\tr\Phi\Phi^*\) vanishes at \(b\in B\) if and only if \(\gamma_b=0\), and \(\gamma\) has exactly \(2g-2-2d\) zeroes (counted with multiplicity) since it is a holomorphic section of a line bundle of degree \(2g-2-2d\).
    \end{proof}
\end{lemma}

Observe that for a genus 2 surface, we only obtain locally homogeneous solitons since \(\mathcal{N}_0\) or \(\mathcal{N}_{1}=\mathcal{N}_{g-1}\) are the only components. However, as soon as the genus of \(B^2\) is \(\ge 3\), we obtain infinite families of solitons that are not locally homogeneous. In fact, by allowing the conformal class \(c\) to vary, we have the following description, from which Corollary~\ref{maincor:dim4} follows easily.

\begin{prop}
    Let \(B^2\) be a closed oriented surface of genus \(g\ge 2\), \(d\in \{0,1\ldots,g-1\}\), and let \(M^4\) be the total space of the oriented vector bundle \(\pi:M^4\to B^2\) with Euler class \(e(M)=d\). Then, there exists a family, depending on \(4(g-1-d)+6(g-1) = 10g-10-4d\) parameters, of invariant expanding Ricci solitons on \(M^4\). These are locally homogeneous if and only if \(d=0,g-1\).
\end{prop}

\subsection{Dimension 5} We now describe the expanding solitons arising from \Cref{Ansatz} with local $\N^3$-symmetry and $2$-dimensional base, where $\N^3$ is the 3-dimensional Heisenberg group (see \Cref{ex:heis3}). Recall that $\N^3$ is the simplest example of a non-abelian nilpotent Lie group.   

\begin{proof}[Proof of \Cref{maincor:dim5}]
We first describe the groups $\G$ and $\K$ associated to $\N^3$ as per \eqref{eqn_G=Autbeta}. The soliton derivation is made explicit in \Cref{ex:heis3}, and it follows that, under the isomorphism $\Aut(\N^3) \simeq \Aut(\ngo)$, we have
\[
    \G \simeq \left\{  
  \begin{pmatrix}
    A & 0 \\
    0 & \det A
  \end{pmatrix} : A \in \GL_2(\R), \,\, \det A = \pm 1 \right\}, \qquad \K \simeq \left\{  
  \begin{pmatrix}
    A & 0 \\
    0 & \det A
  \end{pmatrix} : A \in \mathsf{O}(2) \right\}.
\]

Let $(M^5,g,X)$ be an expanding  soliton constructed from \Cref{Ansatz} with symmetry group $\N^3$. By \Cref{prop_AnsatthenHE}, there is a corresponding solution $(g_B, \rho, h)$ to the twisted harmonic-Einstein equations \eqref{eqn_twistharmEin} for the pair $(B, \G/\K)$. Let $\chi : \pi_1(B) \to \Z/2\Z$ be given by $\chi = {\rm sign} \circ \det \circ \rho$, and let $\hat B \to B$ denote the connected regular covering corresponding to $\ker \chi$. By pulling back $M$ to $\hat B$ if necessary, we assume henceforth that 
\[
    \rho_* : \pi_1(B) \to \SL_2(\R) \subset \G.
\]
Thanks to \Cref{prop_AnsatthenHE} and \Cref{prop:equiv_iff_isom}, to prove \Cref{maincor:dim5} it suffices to give a $1$-to-$1$ correspondence  (up to equivalence) between the triples $(g_B, \rho, h)$ as above, and triples $(g_B^0, \rho_0, h_0)$ solving the twisted harmonic-Einstein equations for the pair $(B, \G^0/\K^0)$, where $\G^0 \simeq \SL_2(\R)$, $\K^0 =\mathsf{SO}(2)$. 

The new monodromy $\rho_0$ is simply the co-restriction of $\rho$ to $\SL_2(\R)$. It is clear that $\rho_0$ is reductive if an only if $\rho$ is so. Regarding $h_0$, note that any element of $\G$ is a product of one in $\G^0$ and another one of the form ${\rm diag}(\pm 1,1,\pm 1) \in \mathsf{O}(3)$. Thus, viewing them as submanifolds of the symmetric space $\SL_3(\R)/\mathsf{O}(3)$, we see that $\G/\K = \G^0/\K^0$. In particular, we may take $h = h_0$, and also $g_B^0 = g_B$, and it is clear that the new triple solves the twisted harmonic-Einstein equations. 

The fact that this association is indeed $1$-to-$1$ up to equivalences follows from the uniqueness statements in \Cref{thm_CorDonLab} and \Cref{Thm:BaseEqn}.
\end{proof}

\begin{remark}
Geometrically, the situation is as follows: there is a globally-defined, free, proper and isometric action of $\R$ on $M^5$, corresponding to the center of $\N^3$ (which is fixed by $\rho : \pi_1(B) \to \SL_2(\R)$). The corresponding quotient $M_0^4 = M^5/\R$ is a twisted principal $\R^2$-bundle, and the quotient metric is an invariant expanding soliton.
\end{remark}

\section{Expanding Ricci solitons with symmetry}\label{Sec:ExpandingSolitons}

In this section, motivated by the asymptotic behavior of Ricci flows on closed manifolds that collapse with bounded curvature and diameter \cites{Lott10,LP25}, we introduce the notion of affine expanding soliton. This generalises in a natural way the concept of algebraic solitons from the homogeneous setting. We prove that affine expanding solitons satisfy \Cref{Ansatz}, and they thus induce solutions to the twisted harmonic-Einstein equations \eqref{eqn_twistharmEin} by \Cref{prop_AnsatthenHE}.  
% show that the Ricci solitons arising as limits of immortal Ricci flows as in \Cref{thm_LP} satisfy \Cref{Ansatz}. Thanks to \Cref{prop_AnsatthenHE}, 
This immediately allows us to  finish the proof of \Cref{main:eRS-tHE}.

\begin{definition}\label{def_affinesoliton}
Let \(\pi:M\to B\) be a $\rho$-twisted principal bundle over a closed manifold \(B\), with \(\rho(\pi_1(B))\le \Aut_{\pm 1}(\N)\). A Ricci soliton \((M,g,X)\) is called an \emph{affine expanding Ricci soliton} if 
\begin{enumerate}
    \item \(g\) is \(\N\)-invariant;
    \item the flow $(\eta_t)_{t\in \mathbb{R}}$ of $X$ consists of twisted principal bundle isomorphisms covering $(\check \eta_t)_{t\in \R} \subset \Diff(B)$;
    \item $(\check \eta_t)_{t\in \R}$ generates a gradient vector field $\nabla f$ on $B$.
\end{enumerate} 
\end{definition}

Conditions (1) and (2) are naturally motivated by making the geometry compatible with the  twisted principal bundle structure. Assumption (3) is justified by  compactness of $B$ (recall that compact Ricci solitons are gradient). We emphasise that all these conditions are satisfied by all asymptotic blow-down limits of Ricci flows that collapse with bounded curvature and diameter \cite{LP25}. In particular, they hold for Lott's examples in the case of abelian symmetry \cites{Lot07,Lott10}.

Taking a time derivative at $t=0$ yields the static equation satisfied by affine expanding solitons: 
\begin{equation}\label{eqn:affinesolRic}
     \Ric_g + {\rm Hess} f + \frac12 g =    
     \begin{pmatrix}
         D_\vca & P \\ P^t & 0
     \end{pmatrix}.
\end{equation}
Here \(f:M\to \R\) is the \(\N\)-invariant lift to $M$, \(D_\vca\in \End(\vca)^{\rm inv}\) is the invariant endomorphism corresponding to a derivation field \(D_{\rm ad}\in \Gamma\Der(\ad M)\), and \(P:\hca\to \vca\) is defined implicitly by \(g(PY,U)= - g([X^\vca,Y],U)\). The blocks in \eqref{eqn:affinesolRic} are according to $TM = \vca \oplus \hca$. If \(\theta:TM \to \vca\) is the connection form for \(\hca\) (see Section~\ref{Sec:Twisted Bundles}), then \(P\) is \(\mathcal{L}_{X_\vca}\theta\) restricted to \(\hca\). (This is seen by applying \(\mathcal{L}_{X_\vca}\) to \(0= \theta Y\), \(Y\in \Gamma(\hca)\).) In particular, \(P\) vanishes if and only if the flow of \(X_\vca\) preserves \(\hca\). 

Notice that the condition that \(\rho(\pi_1(B))\le \Aut_{\pm 1}(\N)\) is equivalent to the existence of a globally-defined volume element for the fibres of $\ad M$ which is $\nabla^\ad$-parallel. This allows us to define the relative volume density $\sqrt {\det h} : B \to \R_+$ of the bundle metric $h$ on $\ad M$. Since we are locally in a principal bundle, the mean curvature vector $\bar \mcv$ satisfies $\pi_* \bar \mcv =: \mcv =  -\nabla \log \sqrt{\det h}$. Combining with the theory developed in \cite{alek_sol} (in the case of principal bundles), this also allows us to consider the $\beta_\ad^+$-weighted volume density of the fibres $v_{\beta^+}^\ad : B \to \R$, see \cite{BL23b}*{$\S$3} for further details about this function. We note that, even though those references are in the context of untwisted principal bundles, the $\Aut(\ngo)$-invariance of the $\beta$-volume density $v_\beta$ and the fact that the structure group of $M$ is contained in the affine group of $\N$, imply that  $v_{\beta^+}^\ad$ is well defined.

In the following theorem, we make use of the `weighted Laplacian' \(\Delta_{\rm m} = \dive_{\rm m}\circ \nabla\). Here ${\rm div}_{\rm m}$ is the divergence operator with respect to the weighted measure 
\[
    {\rm dm} :=  e^{-f} \sqrt{\det h} \, {\rm d vol}_{\check g},
\]
where ${\rm dvol}_{\check g}$ is the Riemannian volume density of $(B,\check g)$. Since $\mcv =  -\nabla \log \sqrt{\det h}$, for any $u\in \mathcal{C}^\infty(B)$ we have
\begin{equation}\label{eqn:weightedLaplacian}
    \Delta_{\rm m} u  = \Delta_{\check g} u   - \check g(\mcv + \nabla f, \nabla u).
\end{equation}

\begin{theorem}\label{thm_afine_then_ansatz}
Let $(M,g,X)$ be an affine expanding Ricci soliton on a $\rho$-twisted principal bundle $\pi : M \to B$ over a closed manifold $B$. Then, $(g,\rho, g_B)$ satisfy the twisted harmonic-Einstein equations \eqref{eqn_twistharmEin} for the pair $(B, \G/\K)$, where $\G$ and $\K$ are defined in \eqref{eqn_G=Autbeta}.
\end{theorem}

\begin{proof}
By \Cref{prop_AnsatthenHE}, it suffices to show that  \Cref{Ansatz} is satisfied. 
% Note that, by construction, we already know that \Cref{Ansatz}, \ref{item:A_flat} holds. 

\smallskip

{\bf \noindent Step 1.} $(D_{\rm ad}, D_{\rm ad}-\tfrac12 g )_{L^2({\rm dm})}\le0$ with equality if and only if the flow of \(X_\vca\) preserves the horizontal distribution \(\hca\).

\smallskip

Let \(\{g_t = \eta^*_tg\}\) be the solution to the normalised Ricci flow \eqref{eqn:NRF} starting from \(g\). Observe that \(g_t\) is also an affine soliton with \(f_t = f\circ \eta_t\). We consider the variation of Perelman's weighted scalar curvature \(\scal_{\rm m} = \scal+2\Delta_g f - |\nabla f|^2\) (cf.~\cite{per1}*{Section 1.3}) along the flow, computing it in two different ways. Notice that the horizontal part of the soliton vector field is \(X_{\hca} = \nabla f\), so that \(\partial_t f_t = |\nabla f|^2\).

%\((g_t,f_t)= (\eta_t^*g,f\circ \eta_t)\), where \(\{\eta_t\}\) is the one-parameter family of diffeomorphisms generated by the soliton vector field \(X = \frac{-1}{2}\nabla f +V\). Since \(g\) and \(f\) are \(\N\)-invariant, \(\scal_{\rm m}\) is too. 

On the one hand, we have \[\scal_{{\rm m},t} = \scal_{\rm m}\circ \eta_t,\] so the variation is \(\m{\nabla \scal_{\rm m}}{\nabla f}\). (Notice that \(\nabla\scal_{\rm m}\) is horizontal since \(\scal_{\rm m}\) is constant on the fibres of \(\pi:M\to B\).) On the other hand, we can calculate using the variational formulae for geometric quantities (see \cite{Besse87}*{\S 1.K}, noting that they use the opposite sign convention for the Laplacian). Equating these gives:
\begin{multline}\label{eq_variationalformula}
    \m{\nabla \scal_{\rm m}}{\nabla f} = \big( \Delta \scal +2\m{\Ric_g +\tfrac{1}{2}g}{\Ric_g}\big) \\ + 2 \big( \Delta |\nabla f|^2  +2 \m{{\rm Hess}f}{\Ric_g +\tfrac{1}{2}g}\big) -\big( 2\Ric(\nabla f,\nabla f) +|\nabla f|^2 +2\m{\nabla |\nabla f|^2}{\nabla f}\big).
\end{multline}
Using the affine soliton equation \eqref{eqn:affinesolRic}, we have \[2\m{\Ric_g +\tfrac{1}{2}g}{\Ric_g} +4\m{{\rm Hess}f}{\Ric_g +\tfrac{1}{2}g}  = 2\m{D_{\rm ad}}{D_{\rm ad}-\tfrac{1}{2}g}+4|P|^2 +\Delta f -2|{\rm Hess} f|^2.\]
%Hence, the right-hand side reduces to \[ -\tfrac{1}{2}\Delta \scal -\Delta |\nabla f|^2 +\Ric(\nabla f,\nabla f) +\tfrac{1}{2}|\nabla f|^2 +\m{\nabla |\nabla f|^2}{\nabla f} -\m{D_{\rm ad}}{D_{\rm ad}-\tfrac{1}{2}g} - \tfrac{1}{2}\Delta f +|{\rm Hess} f|^2.\]
Subbing this into \eqref{eq_variationalformula} and cancelling terms using Bochner's formula, we are left with
\[ \m{\nabla \scal_{\rm m}}{\nabla f} = \Delta \big(\scal +|\nabla f|^2+f\big)  -|\nabla f|^2 -2\m{\nabla |\nabla f|^2}{\nabla f} +2\m{D_{\rm ad}}{D_{\rm ad}-\tfrac{1}{2}g}+4|P|^2  +2\m{\nabla \Delta f}{\nabla f}.\]
Hence, recalling that \(\scal_{\rm m} = \scal +2\Delta f-|\nabla f|^2\) and rearranging gives \begin{align*}
    0& = \Delta \big(\scal +|\nabla f|^2+f\big)  -|\nabla f|^2 -\m{\nabla |\nabla f|^2}{\nabla f} +2\m{D_{\rm ad}}{D_{\rm ad}-\tfrac{1}{2}g}+4|P|^2  -\m{\nabla \scal}{\nabla f}\\
    & =  \Delta_{\rm m} \big(\scal +|\nabla f|^2+f\big)  +2\m{D_{\rm ad}}{D_{\rm ad}-\tfrac{1}{2}g}+4|P|^2.
\end{align*}
The claim now follows.

\bigskip

{\bf \noindent Step 2.} $A=0$, $D_{\rm ad}=\tfrac12 \beta^+_{\ad} \in  \Gamma\Der(\ad M)$ and $\nabla^\ad \beta^+_\ad = 0$.

\smallskip

%(Even though some of these properties are already known from the companion paper \cite{LP25}, we include a proof for the reader's convenience, as we will make use of the rigidity conditions arising from it. )

As in the Einstein case (cf.~\cite{BL23b}), the idea is to apply the maximum principle on $B$ to the function $\log v_{\beta^+}^\ad$,  
but with the `weighted Laplacian' $ \Delta_{\rm m} = {\rm div}_{\rm m} \circ \nabla$. 
By \eqref{Eq:SubmersionCurvature(V)}  and \eqref{eqn:affinesolRic}, we have the following equation on sections of $\End(\ad M)$:
\begin{equation}\label{eqn:Ric_soliton_VV}
    \sum_i (\nabla_{X_i} L)_{X_i} - L_{\mcv+\nabla f}  = \RicE{h} + A^* A +\tfrac12{\rm Id_{ad}}- D_{\rm ad} %\tfrac{1}{2 \tr \beta^2} \beta_\ad . 
\end{equation}
Hence, \eqref{eqn:weightedLaplacian}, \cite{BL23b}*{(3.5) $\&$ Lemma 3.13}, \eqref{eqn:Ric_soliton_VV}  and \Cref{rmk_GIT},  \ref{item_beta+>0}  $\&$ \ref{item_RicGIT} give
\begin{equation}
    \begin{aligned}
     \Delta_{\rm m} \log v^\ad_{\beta^+}
        & \geq   \tr \Big( \big( \sum_i (\nabla_{X_i} L)_{X_i}\big) \beta^+_\ad \Big) - \check g(\mcv + \nabla f, \nabla \log v_{\beta^+}^\ad)  \\
        & =   \tr \Big( \big( \sum_i (\nabla_{X_i} L)_{X_i} -  L_{\mcv + \nabla f}\big) \beta^+_\ad \Big)  \\
        & = \tr \RicE{h} \beta^+_\ad + \tr A^* A \beta^+_\ad +\tfrac12 \tr\beta^+_{\rm ad}- \tr D_{\rm ad} \beta^+_{\rm ad}\\
        & \geq \tfrac12 \tr\beta^+_{\rm ad}- \tr D_{\rm ad} \beta^+_{\rm ad}.  
    \end{aligned}
    \label{eq_estimate}
\end{equation}
Using the properties of \(\beta^+_{\rm ad}\) (see Appendix~\ref{app_nil}), we have
    \[\tfrac12 \tr\beta^+_{\rm ad}- \tr D_{\rm ad} \beta^+_{\rm ad} = 2|\tfrac12 \beta^+_{\rm ad} - D_{\rm ad}|^2 -2\m{D_{\rm ad}}{ D_{\rm ad}-\tfrac12 g}.    \]
Combining this with the estimate \eqref{eq_estimate} gives \[\Delta_{\rm m} \log v^\ad_{\beta^+}\ge 2|\tfrac12 \beta^+_{\rm ad} - D_{\rm ad}|^2 -2\m{D_{\rm ad}}{ D_{\rm ad}-\tfrac12 g}. \]
Integrating this equality over \(B\) with the measure \({\rm dm}\) gives \(D_{\rm ad} = \tfrac12 \beta^+_{\rm ad}\) and \(P=0\) since we have \(-(D_{\rm ad}, D_{\rm ad}-\tfrac12 g)_{L^2({\rm dm)}}\ge 0\) by Step 1. Since \(\tr (\beta^+_{\rm ad})^2 = \tr \beta^+_{\rm ad}\), the estimate \eqref{eq_estimate} becomes 
\[\Delta_{\rm m}\log v_{\beta^+}^\ad \ge 0.\]

By the strong maximum principle, $\log v_{\beta^+}^\ad$ must be constant, and equality occurs everywhere on $B$. Moreover, the rigidity for the last estimate in \eqref{eq_estimate} yields $\beta^+_{\ad} \in  \Gamma\Der(\ad M)$ and $A=0$, whereas that for the first estimate gives $[L_X, \beta^+_\ad] = 0$ for all $X \in TB$. On the other hand, \eqref{eqn_adDecomp} gives 
\[
    \nabla_X^\ad \beta^+_\ad = \nabla_X^{\rm LC} \beta^+_\ad - [L_X, \beta^+_\ad] =  \nabla_X^{\rm LC} \beta^+_\ad ,
\]
where $\nabla^{\rm LC}$ is the connection on $\ad M$ induced by the restriction of the Levi-Civita connection on $M$ to invariant fields. The fact that $\nabla_X^{\rm LC} \beta^+_\ad = 0$ follows by arguing  exactly  as in the proof of \cite{alek_sol}*{Cor.~8.4}.

\bigskip
% \vskip10pt

{\noindent \bf Step 3.} The fibres of $\pi: M \to B$ are nilsolitons and minimal.

\smallskip

For this step we follow the ideas in \cite{alek_sol}*{$\S$9}, which in turn rely heavily on the moment map formulation for the Ricci curvature of nilmanifolds (see \Cref{app_nil}). By Step 2 we have
\begin{align}
    \tau_{\tilde g_B}(h) - L_{\mcv + \nabla f} &=  \RicE{h} - \tfrac12 \beta^+_\ad + \tfrac12 {\rm Id}_{\ad}, \label{eqn:tau=rich} \\
    \Delta_{\rm m} \log v_\beta^\ad &= \tr \big( \beta^+_\ad - {\rm Id}_\ad \big) \big( \tau_{\tilde g_B} - L_{\mcv+ \nabla f} \big). \label{eqn:Deltavbeta}
\end{align}
The convexity properties of the moment map \cite{alek_sol}*{Prop.~9.6} give the following estimates for the scalar curvature of the fibres:
\begin{equation}
    \Delta_{\rm m} \scal(h) \leq -2 \, \tr \Big(  \RicE{h}  \big( \tau_{\tilde g_B} - L_{\mcv + \nabla f} \big) \Big), \label{eqn:Deltascalh}
\end{equation}
with equality if and only if $L_X \in \Gamma \Der(\ad M)$ for all $X\in TB$. Combining \eqref{eqn:tau=rich}, \eqref{eqn:Deltavbeta} and \eqref{eqn:Deltascalh}  we get
\begin{align*}
    \Delta_{\rm m} (\scal h + \log v_\beta^\ad) 
        & \leq  \tr \big( -2\, \RicE{h}  + \beta^+_\ad - {\rm Id}_\ad  \big) \big( \tau_{\tilde g_B}  - L_{\mcv + \nabla f} \big) \\
        &  = -\tfrac12 \, \big| {-2} \, \RicE{h}  + \beta^+_\ad - {\rm Id}_\ad \big|_h^2.
\end{align*}
Again by the strong maximum principle, equality must hold everywhere, and we deduce three things. Firstly, that the fibres are nilsolitons with their induced geometry, as their Ricci curvature satisfies
\[
    \RicE{h} = -\tfrac12 {\rm Id }_\ad + \tfrac12 \beta^+_\ad,
\]
with $\beta^+_\ad \in \Gamma \Der(\ad M)$ (see \Cref{app_nil}). Secondly, equality in \eqref{eqn:Deltascalh} implies that the shape operators are also derivation fields. This has the effect of restricting the image of the bundle metric map $h : \tilde B \to \Sym^2_+(\ngo^*)$ to the orbit of $\Aut(\ngo)$.  Finally, from the fact that $\log v_\beta^\ad$ and $\log v_{\beta^+}^\ad$ are both constant, we must have that $\det h$ is constant as well, hence the fibres are minimal submanifolds.

\bigskip

The results from the previous steps show that \Cref{Ansatz}, \ref{item:A_soliton} holds. Regarding the soliton vector field, the vertical part of $X$ already satisfies \ref{item:A_X}, so  we must show that

\smallskip

{\bf \noindent Step 4. } $X$ is a vertical vector field (i.e. $f$ is constant).

\smallskip 

Notice that, after all the simplifications from the previous steps on the Ricci curvature equations, the data $(\tilde g_B, h)$ of a metric on $\tilde B$ and a map $h: \tilde B \to \G/\K$ satisfy
% vertical and horizontal Ricci curvature equations read as 
\begin{equation}\label{eqn:hRFsoliton}
\begin{split}
    0 & = \tau_{\tilde g_B} h - L_{\nabla f}, \\
    0 & = \Ric_{\tilde g_B} - h^* g_{\rm sym} + {\rm Hess}(f) + \tfrac12 \,  \tilde g_B.
\end{split}
\end{equation}
These are precisely the equations for an expanding gradient soliton for the \emph{harmonic Ricci flow}  (i.e.~Ricci flow coupled with harmonic map flow), see \cite{Mul12}. It was shown in \cite{Wil15}*{Prop.~2.10}, based in turn on Hamilton's analogous Ricci flow result, that for a closed source manifold, expanding gradient solitons must be harmonic-Einstein. The only difference for us is that our source $\tilde B$ is non-compact. However, due to the equivariance of $\tilde g_B$ and $h$, and the fact that $B$ is compact, the same proof  works.  

Indeed, set $\tilde S := \Ric_{\tilde g_B} - h^* g_{\rm sym}$, $\tilde s := \tr_{ \tilde g_B} \tilde S$. From the local computations in \cite{Mul12}*{Thm.~4.5} we have that the evolution of $\tilde s$ along harmonic Ricci flow is 
\[
    \partial_t \tilde s = \Delta \tilde s + 2 \, \| \tilde S \|^2 + 2 \, \| \tau_{\tilde g_B}(h) \|^2.
\]
On the other hand, given \eqref{eqn:hRFsoliton} it follows that the solution is self-similar, thus 
\[
    \partial_t \tilde s =  \widetilde X_\hca (\tilde s) - \tilde s,
\]
where $\widetilde X_\hca$ is the lift of $\nabla f$ to $\tilde B$. Since $\tilde s$ is $\pi_1(B)$-invariant, it induces a smooth function $s: B \to \R$, which by the above satisfies
\[
    \Delta s - \langle \nabla  s , \nabla f  \rangle  +  2 \, \|  S \|^2 + 2 \, \| \tau_{\tilde g_B}(h) \|^2 + s = 0.
\]
At this point one may argue exactly as in the Ricci flow case, see e.g.~the proof of \cite{libro}*{Prop.~1.13}, to conclude that $f$ is constant.

\bigskip

{\bf \noindent Step 5. } The monodromy of \(\mathcal{H}\) is \(\rho\).

\smallskip
Our goal is to show that the monodromy has no translational part. Observe that we already have this monodromy reduction on an infinitesimal level. Namely, the existence of the parallel section \(\beta^+_{\rm ad}\) of \(\End(\ad M)\) implies that \(\rho_{\hca,*}(\pi_1(B))\le \Aut(\mathfrak{n})\) centralises \(\beta^+_h\) by the holonomy principle. The general monodromy reduction will follow similarly: It follows from the previous steps that the affine isomorphisms \(\{\eta_t\}\) driving the Ricci flow preserve the connection. (We have equality in Step 1 since \(D_{\ad}=\tfrac{1}{2}\beta^+_{\ad}\), so the flow of \(X=X_\vca\) preserves \(\hca\).) Hence, their restriction to a given fibre of \(\pi:M\to B\) commutes with the monodromy. This will force the modromy to consist of automorphisms by Lemma~\ref{lem_commutation}.

Fix \(b\in B\). By Steps 3 and 4, the fibre \(M_b=\pi^{-1}(b)\) with the vector field \(X|_{M_b}\) is a nilsoliton for the induced metric. In particular, there is a \(p\in M_b\) such that \(X_p=0\) by \Cref{prop_nilsoliton_vf}. If \(\hat \gamma\) is the horizontal lift of a closed curve representing \(\gamma \in \pi_1(B)\) beginning at \(q\in M_b\), then \(\eta_t\circ \hat\gamma  \) is a horizontal lift starting from \(\eta_t(q)\). Hence, we have 
\[\eta_t(\rho_{\hca }(\gamma)q) = \eta_t(\hat\gamma(1)) = (\eta_t\circ\hat\gamma)(1) =  \rho_{\hca}(\gamma)\eta_t(q).\]
Identifying \(M_b\simeq \N\) with an affine map that sends \(p\in M_b\) to \(e\in \N\), \(\{\eta_t|_{M_b}\}\) is identified with a one-parameter family of automorphisms (since it is a one-parameter family of affine maps fixing the identity). Applying Lemma~\ref{lem_commutation}, we see that \(\rho_{\hca}(\gamma)\in \Aut(\N)\) for all \(\gamma\in \pi_1(B)\). Since the automorphism part of the monodromy is precisely \(\rho\), the claim follows.

\end{proof}

\begin{proof}[Proof of \Cref{main:eRS-tHE}]
    Given a solution \((g_B,\rho,h)\) to the twisted harmonic-Einstein equation for the pair \((B,\G/\K)\), \Cref{prop_AnsatthenHE} constructs an expanding Ricci soliton  \((M:=\tilde B\times_\rho \N, g, X)\) satisfying \Cref{Ansatz}. Since $X$ is vertical, its flow consists of bundle maps, and restricted to each fibre it is the flow of a nilsoliton. By \Cref{prop_nilsoliton_vf}, the fibre restrictions are affine maps, thus the soliton is affine. Conversely, if \((M,g,X)\) is an affine expanding soliton, then \((g,X)\) satisfies Ansatz~\ref{Ansatz} by Theorem~\ref{thm_afine_then_ansatz}. Hence, we obtain a solution to the twisted harmonic-Einstein equation by \Cref{prop_AnsatthenHE}.
\end{proof}

To conclude this section, we now justify some claims made in the introduction.

\begin{prop}\label{prop:non-gradient}
An affine expanding Ricci soliton is of gradient type if and only if $\N\simeq \R^n$ is abelian and $\rho$ takes values in $\mathsf{O}(n)$.
\end{prop}
\begin{proof}
Assume that $X = \nabla f$ for some smooth function $f \in \mathcal{C}^\infty$. 
By \Cref{Ansatz}, the structure group of the bundle $\pi: M\to B$ reduces to $\Aut_{\pm 1}(\N)$, thus $M$ is a Lie group bundle (vector bundle in case $\N$ is abelian), and there is a well-defined identity (zero) section. From \Cref{Ansatz}, \ref{item:A_X} we also know that $X$ is vertical and it restricts to a soliton vector field $X_b$ of each fibre $M_b$. In particular, $f$ is horizontally constant. Moreover, by the construction in \Cref{Lemma:SolVec}, $X$ vanishes along the identity section.  Notice that $X$ is well-defined only up to adding a global Killing field $K$, which in the gradient case must be parallel. We will assume $K=0$; the proof for $K\neq 0$ parallel is analogous.

Suppose first that $\N$ is abelian. We may assume without loss of generality that $f$ vanishes along the zero section.  The fibre $M_b$ is a flat Gaussian soliton, and its vector field is the Euler vector field (the one whose flow is given by scalar multiplication). In terms of the potential, we have 
\[
 f(b,v) = \tfrac12 |v|_{h(b)}^2,  \qquad b\in U \subset B, \,\, v\in \R^n,
\]
in a local trivialisation $U\times \R^n$ adapted to the flat connection. Extending $v$ to be covariantly parallel in horizontal directions, and using that $f$ is horizontally constant, we deduce that $h$ is preserved by the flat connection. This in particular implies that the monodromy $\rho$ takes values in $\mathsf{O}(n)$. The converse is clear.

If $\N$ is nilpotent and  non-abelian, it suffices to show that the fibres are non-gradient solitons. This is well-known, and it follows from rigidity of homogeneous gradient solitons \cite{PW}. Indeed,  any homogeneous gradient soliton is isometric to (a quotient of) $\R^k \times E^{n-k}$, where $E$ is Einstein and $\R^k$ flat.  In particular, its Ricci curvature is semi-definite, and this contradicts a result of Milnor \cite{Mln}. 
\end{proof}

\begin{remark}\label{rmk_nongrad}
\Cref{prop:non-gradient} also shows that, even in the abelian case, many of the solitons we constructed are of non-gradient type. Indeed, Hitchin  \cite{Hit92} showed that, for $n>2$, the representations with $[\rho] \in \Hom(\pi_1(B), \mathsf{PSL}_n(\R))/\mathsf{PSL}_n(\R)$ which can be continuously deformed to one taking values in the maximal compact subgroup $\mathsf{PSO}(n)$ all belong to a single connected component of the character variety (whose elements have the same topological type as those in the Hitchin component). But there are 2 (resp.~5) other connected components when $n$ is odd (resp.~even).
\end{remark}

\begin{remark}\label{rmk_not_conical}
It follows from \eqref{eqn:affinesolRic} and the proof of \Cref{thm_afine_then_ansatz} that  the scalar curvature of a soliton $(M,g,X)$ constructed from \Cref{Ansatz} is a negative constant, given by
\[
    \scal_g = \begin{cases}
        - (\dim B)/2, & \qquad \N \simeq \R^n \hbox{ abelian;}\\
        -1 / (\tr \beta^2), & \qquad \hbox{otherwise.}
    \end{cases}
\]
In particular, $(M,g)$ is \emph{not} asymptotically conical, since the scalar curvature would otherwise decay to $0$ at infinity. 
\end{remark}

\section{Einstein extensions\texorpdfstring{: proof of \Cref{main_Einstein}}{}}\label{sec:Einstein}

In this section we show that the expanding Ricci solitons constructed via \Cref{Ansatz} give rise to an extension, one dimension higher, which is Einstein. 

More precisely, let $\pi : M \to B$ be a twisted principal bundle such that $(M,g,X)$ is a solution to the soliton equation \eqref{eqn_eRS} obtained from \Cref{Ansatz}. Consider the derivation field of $\ad M$ obtained as in the proof of \Cref{prop_AnsatthenHE} from the Ricci curvature of the fibres. It gives rise to an invariant endomorphism field $D_M \in \End(TM)$ as in \Cref{Lemma:SolVec}, which is related to the Ricci endomorphism of the soliton via 
\begin{equation}\label{eqn:Ricg}
    \RicE{g} = -\tfrac12 \, {\rm Id} - D_M. 
\end{equation}
To define the extension, we follow the ideas and notation in \cite{NikoAlek21}, which are in turn inspired in the homogeneous case. Namely, let $M_0 := \R \times M$, set $D := \alpha D_M \in \End(TM)$, with $\alpha > 0$ a constant to be determined later. Endow $M_0$ with the metric $du^2 + g_u$, where $g_u := \exp (u  \, D)^* g$ is a one-parameter family of metrics in $M$. By \cite{NikoAlek21}*{Thm.~1.3}, $(\R\times M, du^2 + g_u)$ is Einstein (with Einstein constant $-\tr(D^2)$) if and only if the following conditions hold:
\begin{enumerate}  
    \item[(i)]  The endomorphism $D$ has constant eigenvalues and ${\rm div}_g D = 0$;
    \item[(ii)] The Ricci endomorphism of $g_u$ is independent of $u$ and is given by 
    \[
            g_u^{-1}\Ric_{g_u} = (\tr D) D - \tr (D^2) {\rm Id}.
    \]
\end{enumerate}

Since $\exp(u D_M) : (M, g_u) \to (M,g) $ is an isometry, and $\exp(u D_M)$ pointwise commutes with $D_M$, condition (ii) reduces to a verification for $u=0$, i.e.~for $g$. Recall that  $D_M$ satisfies $\tr (D_M^2) = -\tfrac12  \tr D_M$ \eqref{eqn:trD2}. Choosing $\alpha = (2 \tr (D_M^2))^{-1/2}$ we get 
% \[
%     \tr D_M^2 = -\tfrac12 \tr D_M, \qquad (\tr D) D = \alpha^2 \, (\tr D_M) \, D_M
% \]
% \[
%     \alpha^2 \tr D_M = -1, \qquad \alpha^2 \tr D_M^2 = \tfrac12 
% \]
% }
\[
    \tr(D^2) = \frac12, \qquad (\tr D) D = -D_M.
\]
Thus, (ii) now follows from  \eqref{eqn:Ricg}. On the other hand, it is clear that $D_M$ has constant eigenvalues, thus $g$ has constant scalar curvature. Hence, by \eqref{eqn:Ricg} and the contracted second Bianchi identity, we have 
\[
    -{\rm div } D_M  =  {\rm div} (\RicE{g}) = {\rm div} \Ric_g = \tfrac12 \, d \scal_g = 0.
\]
This concludes the proof of \Cref{main_Einstein}.

\appendix

\section{Nilsolitons}\label{app_nil}

In this section we collect some useful information about nilsolitons which we have used throughout the article. Some references for this topic are \cite{soliton}, \cite{cruzchica}. Here, $\N$ will denote a simply-connected nilpotent Lie group.

\begin{definition}
A \emph{nilsoliton} is a left-invariant Ricci soliton metric $h$ on $\N$, i.e.~its Ricci curvature satisfies
\[
    \Ric_h = \lambda h -\tfrac12 \mathcal{L}_X h,
\]
for some vector field $X$ on $\N$ (a \emph{soliton vector field}), and some  $\lambda \in \R$ (the \emph{cosmological constant}).
\end{definition}

The following was originally shown in \cite{soliton}, with a correction in \cite{Jbl2015}:

\begin{prop}\label{prop_nilsoliton}
A left-invariant metric $h$ on $\N$ is a nilsoliton if and only if its (1,1)-Ricci tensor restricted to the Lie algebra $\ngo$ of left-invariant vector fields on $\N$ satisfies 
\begin{equation}\label{eqn:algsoliton}
    h^{-1} \Ric_{h} = \lambda \, {\rm Id}_\ngo - D, \qquad D\in \Der(\ngo).
\end{equation}
\end{prop}
\begin{proof}
Without loss of generality let us assume $\lambda = -1/2$. If $h$ is a left-invariant Ricci soliton, let $(\eta_t)_{t\in \R}$ be the flow of $X$. Then $h_t := \eta_t^* h$ solves the normalised Ricci flow \eqref{eqn:NRF}. Since Ricci flow preserves isometries, $\N$ acts by isometries of $h_t$ for all $t\in \R$. By Wilson's theorem \cite{Wil82}, $\eta_t \in \Aut(\N)\ltimes \N$, thus there exists a smooth one-parameter family $(\varphi_t)_{t\in \R}\subset \Aut(\N)$ such that $h_t = \varphi_t^* h$. We may thus assume that $X$ is generated by $\varphi_t$. Since $\N$ is simply-connected, there exists $\Phi_t \in \Aut(\ngo)$ with $d \Phi_t |_e = \varphi_t$. Write $\Phi_t = \exp(t D)$ for some $D \in \Der(\ngo)$. Then, at the identity $e\in \N$ it holds that
\[
    (\varphi_t^* h)_e = h(\exp(tD) \cdot, \exp(tD) \cdot ),
\]
from which $(\mathcal{L}_X h)_e (\cdot, \cdot) = h(D \cdot, \cdot) + h(\cdot, D \cdot)$. It follows that 
\[
    \RicE{h}  = -\tfrac12 {\rm Id} - \tfrac12 (D+D^T).
\]
By \Cref{lem:D_perp_adU} below this implies $\tfrac12 (D+D^T) \in \Der(\ngo)$. Replacing $D$ by $\tfrac12 (D+D^T)$, \eqref{eqn:algsoliton} follows. 

Conversely, assume \eqref{eqn:algsoliton} holds, and let $\varphi_t \in \Aut(\N)$ be determined by $\exp(tD) \in \Aut(\ngo)$. By the same  computations as above, the Ricci soliton equation is satisfied for the vector field $X$ generated by $\varphi_t$.
\end{proof}

\begin{example}\label{ex:heis3}
If $\N$ is the three-dimensional Heisenberg Lie group of $3\times 3$ upper-triangular real matrices with $1$'s on the diagonal, then any left-invariant metric is a nilsoliton. Indeed, from  \cite{Mln} there is an orthonormal basis $\{e_1,e_2,e_3\}$ of $\ngo$ with $[e_1, e_2] = 2  c \,  e_3$, and the Ricci endomorphism is given by 
\[
   \RicE{h} =  -3c^2 \, {\rm Id} + 2c^2 \, {\rm diag}(1, 1, 2), \qquad{\rm diag}(1, 1, 2) \in \Der(\ngo).
   % \begin{pmatrix}
   %       -3c^2 & & \\
   %       & -3c^2 & \\
   %       & & -3c^2
   %  \end{pmatrix}
   %  + 
   %  \begin{pmatrix}
   %       2c^2 & & \\
   %       & 2c^2 & \\
   %       & & 4c^2
   %  \end{pmatrix}.
\]
\end{example}

We call the derivation $D\in \Der(\ngo)$ from \eqref{eqn:algsoliton} the \emph{soliton derivation}. This derivation has a number of properties arising from the GIT applications to this setting. Among them, the fact that $-D$ is positive definite is of fundamental importance for most applications (see \Cref{rmk_GIT}, \ref{item_beta+>0} below). 

Let $X_D$ be the vector field on $\N$ whose flow is the one-parameter subgroup of automorphisms $(\eta_t) \subset \Aut(\N)$ satisfying $(\eta_t)_* = \exp(tD)$ for all $t\in \R$. Notice that $X_D$ is nothing but the fundamental vector field corresponding to $D$ for the action of $\Aut(\N)$ on $\N$. This is an example of a \emph{linear vector field}, since in exponential coordinates $\exp : \ngo \to \N$, $X_D$ is simply given by $\ngo \ni U  \mapsto DU$. 

\begin{prop}\label{prop_nilsoliton_vf}
Let $(N,h)$ be a nilsoliton with soliton derivation $D$. Then, any soliton vector field $X$ is of the form $X_{\tilde D} + Y_R$, where $Y_R$ is any right-invariant vector field on $\N$, and $\tilde D \in \Der(\ngo)$ is such that $\tfrac12 (\tilde D + \tilde D^T) = D$. Moreover', any such $X$ has a unique zero on $\N$. 
\end{prop}
\begin{proof}
By the proof of \Cref{prop_nilsoliton}, it follows that $X_D$ is a soliton vector field. Any two soliton vector fields differ by a Killing field, and by \cite{Wil82}, for a left-invariant metric on a nilpotent Lie group, these are always of the form $ X_{D_1} +  Y_R$, where $Y_R$ is right-invariant and $X_{D_1}$ is linear, with $D_1 \in \Der(\ngo)$ a $h$-skew-symmetric derivation. Hence, any soliton vector field is of the form $X_D + X_{D_1} + Y_R = X_{\tilde D} + Y_R$, $\tilde D:= D + D_1$, and the first part of the statement follows.

Applying  $\mathcal{L}_{X_{D_1}}$ to the soliton equation yields
\[
    0 = \mathcal{L}_{X_{D_1}} \mathcal{L}_{X_D} h =  \mathcal{L}_{[X_{D_1}, X_D]} h = \mathcal{L}_{X_{[D_1, D]}} h.
\]
Then $X_{[D_1, D]}$ is a Killing field vanishing at $e$, thus of the form $X_{D_2}$ with $D_2$ a skew-symmetric derivation. Thus, $[D_1, D] = D_2$, and since the left-hand-side is symmetric, we have $[D_1, D] = 0$. In particular, $\tilde D$ is invertible (its symmetric part $D$ is positive definite).
% whose symmetric part $D$ is positive definite, and hence $\tilde D$ is invertible  as well. 

% To see this, notice that if $X_D + X_{D_1} + Y_R$ is a soliton vector field, then so is $X_D+X_{D_1}$, by left-invariance of $h$. Precisely as in the proof of \Cref{prop_nilsoliton} it follows that $(D+D_1)$ This implies that $X_{D} + X_{D_1} = X_{\tilde D}$ for $\tilde D = D + D_1$, and the first claim follows. 

Regarding the zeroes of $X$, we first notice that $\tilde D$ is invertible, as it is a normal operator with positive-definite self-adjoint part. We work in exponential coordinates. Using the Baker-Campbell-Hausdorff formula and nilpotency, it is easy to see that a right-invariant vector field $Y_R$ whose value at $e$ is $Y\in \ngo$ is represented in these coordinates by a map 
\[
    \ngo \ni U \mapsto  F(U) := P(\ad U)Y,
\]
where $P$ is a polynomial whose constant coefficient is $1$. Finding a zero for $X = X_{\tilde D} + Y_R$ is thus equivalent to proving that the map
\begin{equation}\label{eqn_DU+FU=0}
 \tilde D U + F(U) = 0
\end{equation}
always has a solution on $\ngo$. To prove this, we work iteratively with the lower central series $\ngo^i := [\ngo, \ngo^{i-1}]$, $\ngo^1 = \ngo$. We will construct, by induction on $k$, a $U_k$ such that $U_k + F(U_k) \in \ngo^{k+1}$. For the base case, we look at \eqref{eqn_DU+FU=0} modulo $\ngo^2$. Since $\tilde D$ is a derivation, it preserves $\ngo^2$, and it hence induces an invertible map on $\ngo^1/\ngo^2$. On the other hand, $F(U)$ modulo $\ngo^2$ is simply the image of $Y$ in $\ngo^1/\ngo^2$, therefore we can always find $U_1$ solving \eqref{eqn_DU+FU=0} modulo $\ngo^2$. 

Assuming we have constructed $U_k$ solving \eqref{eqn_DU+FU=0} modulo $\ngo^{k+1}$, we now seek for $U_{k+1}$ of the form $U_k + V$ with $V\in \ngo^{k+1}$. We have 
\[
    \tilde D U_{k+1} + F(U_{k+1}) = \tilde DU_k + F(U_k) + \tilde D V + (F(U_{k+1}) - F(U_k)).
\]
Since $V\in \ngo^{k+1}$, and all terms in $F(U_k + V) - F(U_k)$ involve at least one bracket with $V$, it follows that this last term is in fact in $\ngo^{k+2}$, and so it can be ignored. By the invertibility of the induced map $\tilde D$ on $\ngo^{k+1}/\ngo^{k+2}$, and the fact that $\tilde D U_k + F(U_k) \in \ngo^{k+1}$, we can solve for $V \in \ngo^{k+1}$ so that 
\[
    \tilde D U_k + F(U_k) + \tilde D V \in \ngo^{k+2},
\]
and this completes the proof.
\end{proof}

\begin{corollary}\label{cor_solitonderiv}
Let $(\N,h)$ be a nilsoliton with soliton derivation $D$. Then, the soliton vector field may be chosen to be $X_D$, the unique vector field whose flow is a one-parameter subgroup of automorphisms $(\eta_t) \subset \Aut(\N)$ satisfying $(\eta_t)_* = \exp(tD)$ for all $t\in \R$.
\end{corollary}

\begin{lemma}\label{lem_commutation}
Let $(\N,h)$ be a nilsoliton with soliton derivation $D$, and let $(\eta_t)_{t\in \R} \subset \Aut(\N)$ be as in  \Cref{cor_solitonderiv}. If $\varphi \in \Aut(\N), x\in \N$ are such that $\varphi \circ L_x$ commutes with $\eta_t$ for all $t\in \R$, then $x = e$. 
\end{lemma}
\begin{proof}
It suffices to show that $x = \eta_t(x)$ for all $t \in \R$. Indeed, since $\N$ is nilpotent and simply-connected, there exists $X\in \ngo$ such that $x = \exp(X)$. Then $X = (\eta_t)_* X$, and differentiating at $t=0$ we get $DX =0$, from which $X=0$ by definiteness of $D$. 

Now recall that for any $\psi\in \Aut(\N)$ we have  $\psi \circ L_x = L_{\psi(x)} \circ \psi$. Hence, 
\[
    \varphi \circ L_x \circ  \eta_t = \eta_t  \varphi \circ L_x = \eta_t  \varphi \eta_t^{-1} \eta_t  \circ L_x 
    = \eta_t \varphi \eta_t^{-1} L_{\eta_t(x)} \circ \eta_t,
\]
from which we quickly deduce that
\[
    L_{x \eta_t(x)^{-1}} = \varphi^{-1} \eta_t \varphi \eta_t^{-1} \in \Aut(\N).
\]
Evaluating at $e$ yields $x \eta_t(x)^{-1} = e$ for all $t\in \R$, and this concludes the proof.
\end{proof}

It is known that every simply-connected nilpotent Lie group of dimension up to $6$ admits a nilsoliton metric \cites{finding,Wll03}, but already in dimension $7$ there are some which do not: see \cite{einsteinsolv} for details about the existence and non-existence question. In general, the classification of nilpotent Lie algebras admitting nilsoliton inner products is wide open as far as we know.

\smallskip

While these metrics do not exist on any nilpotent Lie group, when they do, they are unique up to isometry, in a very strong sense: given any two nilsolitons $h_1$, $h_2$ on the same Lie group $\N$, up to scaling there exists $\varphi\in \Aut(\N)$ such that $h_2 = \varphi^* h_1$ \cite{soliton}. 

A left-invariant metric on a connected Lie group $\N$ is determined by an inner product $h\in\Sym_+^2(\ngo^*) $ on the Lie algebra $\ngo$.  The space of inner products  is a Riemannian symmetric space of non-positive curvature: 
\[
\Sym^2_+(\ngo^*) = \Glp(\ngo) / \mathsf{SO}(\ngo, \bar h). 
\]
Here $\bar h$ is certain fixed inner product, and the action of $\Glp(\ngo)$ is by pull-back: 
\begin{equation}\label{eqn:GLnacts_sym2}
    (q\cdot \bar h) (\cdot ,\cdot) = \bar h (q^{-1} \cdot, q^{-1} \cdot), \qquad q\in \Glp(\ngo).    
\end{equation}
The symmetric metric on $\Sym^2_+(\ngo^*)$ is  given for $h\in \Sym^2_+(\ngo^*)$ and  $k\in T_{h} \Sym^2_+(\ngo^*) = \Sym^2(\ngo^*)$ by
\begin{equation}\label{eqn:g_sym}
    (g_{\rm sym})_h (k,k) = \sum_{i,j = 1}^{\dim \ngo} k(e_i, e_j) k(e_i,e_j), 
\end{equation}
for any $h$-orthonormal basis $\{e_i\}$ of $\ngo$. See \cite{Heber1998}*{$\S$3.1} for further details.

In terms of this structure, if $\bar h$ is itself a nilsoliton with cosmological constant $\lambda$, then by the above-mentioned uniqueness, the set of nilsoliton inner products on $\ngo$ with cosmological constant $\lambda$ is the orbit 
\[
    \Aut(\ngo) \cdot \bar h \,\,  \subset \,\,  \Glp(\ngo)/\SO(\ngo,\bar h).
\]
Set $\beta = \RicE{\bar h} \in \End(\ngo)$, let $\bar \nu \in \Lambda^{\rm top} \ngo^*$ correspond to the Riemannian volume form of $\bar h$, and denote by $|\bar \nu |$  the corresponding volume density.

\begin{lemma}\label{lem:solitons_connected}
The set of nilsoliton inner products in $\Sym^2_+(\ngo^*)$ with Ricci endomorphism  $\beta$ and with volume density $|\bar\nu|$ is connected, and it is given by the orbit
\[
    \G \cdot \bar h \,\,  \subset \,\,  \Glp(\ngo)/\SO(\ngo,\bar h), \qquad
   \G := \{ \varphi \in \Aut(\ngo) : \varphi \beta \varphi^{-1} = \beta, \,\, \det \varphi = \pm 1 \}.
\]
\end{lemma}
\begin{proof}
% We first claim that the set of such nilsolitons is the orbit 
% \[
%     \Aut_1(\ngo)^\beta \cdot \bar h \,\,  \subset \,\,  \Glp(\ngo)/\SO(\ngo,\bar h).
% \]
% Indeed, 
By uniqueness, nilsoliton inner products lie in the orbit $\Aut(\ngo) \cdot \bar h$. Each $\varphi \in \Aut(\ngo)$ gives an isometry of metric Lie algebras $\varphi : (\ngo, \bar h) \to (\ngo, \varphi \cdot \bar h)$, thus the Ricci endomorphism of $\varphi\cdot \bar h$ is $\varphi \beta \varphi^{-1}$,  and its volume density is $|{\det \varphi}|^{-1} |\bar \nu|$. 

To see that $\G \cdot \bar h$ is connected, we use that $\G$ is real reductive, that is, it admits a Cartan decomposition
\[
    \G =  \exp (\mathfrak{p}) \K, \qquad \K :=  \G \cap \mathsf{O}(\ngo, \bar h),  \quad \mathfrak{p} := \{ E\in \Der(\ngo) : [E,\beta] = 0, E = E^T, \tr E = 0 \}.
\]
This fact is well-known to experts, see \cite{Nik11}*{Thm.~5.1} and \cite{RS90}*{Thm.~4.3}. It follows that $\G \cdot \bar h = \{ \exp(E) \cdot \bar h : E\in \mathfrak{p} \}$ is an embedded, connected submanifold of $\Glp(\ngo)/\SO(\ngo,\bar h)$.
% In particular, the orbit of the identity component $\G := (\Aut_1(\ngo)^\beta)^0$ through $\bar h$ equals $\Aut_1(\ngo)^\beta \cdot \bar h$. 
\end{proof}

A key ingredient in the study of nilsolitons is the following moment map formulation for the Ricci endomorphism  of a left-invariant metric $h$ on $\N$ (see e.g.~\cite{Lau06}*{Prop.~4.2}; to see in what sense this is a moment map equation see \cite{GIT20}*{$\S 1.3$}.):
\begin{equation}\label{eqn:Ric_mm}
    \tr \big( (h^{-1} {\Ric_h}) E \big) = -\tfrac12 \langle  \delta (E), [\cdot,\cdot]\rangle_{ h},\qquad \forall \, E\in \End(\ngo).
\end{equation}
Here $[\cdot,\cdot] \in \Lambda^2\ngo^* \otimes \ngo$ is the Lie bracket of $\ngo$,  $\delta(E)(\cdot,\cdot) = -E [\cdot,\cdot] + [E\cdot, \cdot] + [\cdot, E\cdot] \in \Lambda^2\ngo^* \otimes \ngo$, and the inner product is the natural extension of $h$ to the $\Lambda^2\ngo^* \otimes \ngo$. Notice that $\delta(D) = 0$ precisely when $D\in \Der(\ngo)$. 

As an immediate application of \eqref{eqn:Ric_mm} we see that the Ricci endomorphism is orthogonal to derivations, thus the  nilsoliton derivation $D$  from \Cref{prop_nilsoliton} must satisfy
\begin{equation}\label{eqn:trD2}
    \tr D^2 = \lambda \tr D.
\end{equation}
Also, taking $E = h^{-1}{\Ric_h}$ in \eqref{eqn:Ric_mm} yields 
\[
    |{\Ric_h}|^2 = \lambda \, \scal_h .
\]
Since $\scal_h \leq 0$ with equality if and only if $\ngo$ is abelian \cite{Mln}, it follows that non-abelian nilsolitons are of expanding type $\lambda < 0$. 

Another application of \eqref{eqn:Ric_mm} is the following lemma about derivations that commute with Ricci:

\begin{lemma}\label{lem:D_perp_adU}
Let $(\N,  h)$ be a nilpotent Lie group with invariant metric, with Ricci endomorphism $ h^{-1}\Ric_{ h} = \tfrac12 \beta$. If $D\in \Der(\ngo)$ is such that $[D+D^T,\beta]= 0$ (this holds in particular if $[D,\beta] = 0$), then $L := D+D^T \in \Der(\ngo)$ and 
\[
    \tr \big((\ad U) L \big) = 0, \qquad \forall  \, U\in \ngo. 
\]
\begin{proof}
Taking $E = [D,D^T]$ in \eqref{eqn:Ric_mm} (transpose with respect to $\bar h$), standard properties of $\delta$  and the extended inner products (see \cite{alek}*{$\S$7}) give 
% \begin{align*}
\[
    0    = \tr  [\beta, D+D^T] D^T = \tr \beta [D,D^T]  = -\tfrac12  \langle \delta([D,D^T]) , [\cdot,\cdot] \rangle_{\bar h}  = -\tfrac12 \| \delta(D^T) \|_{\bar h}^2.
\]
It follows that $D^T \in \Der(\ngo)$, and the same for $L$. To see that the trace against $\ad U$ vanishes, consider the orthogonal decomposition $\ngo = \ngo_1 \oplus \cdots \oplus \ngo_k$, where $\ngo_2 \oplus \cdots \ngo_k = [\ngo,\ngo]$, $\ngo_3\oplus \cdots \ngo_k = [\ngo,[\ngo,\ngo]]$, and so on. It is easy to see that, in an adapted basis, the matrix of $\ad U$ is strictly block-lower-triangular, whereas the matrix of the derivation $L$ is block-lower-triangular, and the lemma follows. 
\end{proof}
\end{lemma}

\begin{remark}\label{rmk_GIT}
Formula \eqref{eqn:Ric_mm} is in fact just the tip of the iceberg, giving rise to an extremely fruitful application of ideas from Geometric Invariant Theory, extended to the case of real reductive Lie groups \cites{RS90,Mar2001, HS07, HSS}, to problems involving the Ricci curvature of homogeneous Riemannian manifolds, and more generally Riemannian manifolds with symmetries, see e.g. \cites{Heber1998,standard,JblPet14, BL17,alek_sol}. Many of these applications rely on a Kirwan-Ness stratification of the variety of Lie algebras. One associates to each metric Lie algebra $(\ngo, h)$ a uniquely-determined $h$-self-adjoint endomorphism $\beta_h \in \End(\ngo)$, called the \emph{GIT weights}. Its conjugacy class is independent of $h$ and serves as a label for the stratum containing the Lie algebra $\ngo$.

We now list some properties of the GIT weights that are key in some of the proofs in this article. We refer the reader to \cite{GIT20} and therein references for proofs and further details.
\begin{enumerate}[label=(\roman*)]
    \item \label{item_trbeta} The conjugacy class of $\beta_h$ is independent of $h$. In particular, $\tr \beta_h^2 = : \tr \beta^2$ is constant, and we normalise $\tr \beta_h = - 1$;
    \item \label{item_beta+>0} $\beta^+_h := {\rm Id} + (\tr \beta^2)^{-1} \beta_h$ is positive definite, provided $\ngo$ is nilpotent;
    \item \label{item_RicGIT} $\tr \big(\RicE{h} \beta^+_h) \geq 0$, with equality if and only if $\beta^+_h\in \Der(\ngo)$;
    \item \label{item_soliton} For any inner product $h$ on $\ngo$ with $\scal_h = -1$ we have that $|{\Ric_h}|_h^2 \geq \tr \beta^2$, with equality if and only if $h$ is a nilsoliton and $h^{-1}\Ric_h = \beta_h$.
\end{enumerate}
\end{remark}

% \bibliography{ramlaf2}
\bibliography{../ramlaf2}
\bibliographystyle{amsalpha}

\end{document}